\documentclass[11pt]{article}

\usepackage{amsmath,amsthm,amscd,latexsym}
\usepackage{amsfonts,pdfsync,color,graphicx}
\usepackage[psamsfonts]{amssymb}
\usepackage{epsfig}
\usepackage{color}
\usepackage[noadjust]{cite}
\usepackage{accents}
\usepackage{caption}
\usepackage{listings}
\usepackage{multicol}
\usepackage{float}
\usepackage{hyperref}
\usepackage{tikz-cd}
\usepackage{subfigure}

\usepackage{listings}
\usepackage{xcolor}

\lstdefinestyle{mathematica}{
	language=[Mathematica]Mathematica,
	basicstyle=\ttfamily\small,
	keywordstyle=\color{blue}\bfseries,
	commentstyle=\color{green!50!black}\itshape,
	stringstyle=\color{orange},
	numbers=left,
	numberstyle=\tiny\color{gray},
	stepnumber=1,
	numbersep=5pt,
	frame=single,
	breaklines=true,
	tabsize=2,
	showstringspaces=false
}

\usepackage[top=1.5in,bottom=1.5in,left=1.3in,right=1.0in]{geometry}

\newtheoremstyle{mystyle}               
{}                
{}                
{}        
{}                
{\bfseries \itshape}       
{.}      
{ }      
{}       

\newtheorem{theorem}{Theorem}[section]
\newtheorem{proposition}[theorem]{Proposition}
 
\newtheorem{corollary}[theorem]{Corollary}
\theoremstyle{definition}
\newtheorem{definition}[theorem]{Definition}
\newtheorem{example}[theorem]{Example}	
\newtheorem{conjecture}[theorem]{Conjecture}	
\theoremstyle{mystyle}
\newtheorem{remark}[theorem]{Remark}
\numberwithin{equation}{section}

\newcommand{\norm}[1]{\left\lVert #1 \right\rVert}

\title{Second order periodic boundary value problems with reflection and piecewise constant arguments}
\date{ }
\author{Alberto Cabada and Paula Cambeses-Franco\\
	$^1$ CITMAga, 15782, Santiago de Compostela, Galicia, Spain\\
	$^2$ Departamento de Estatística, Análise Matemática e Optimización\\
	Facultade de Matem\'aticas, Universidade de Santiago de Com\-pos\-te\-la, Spain.\\
	alberto.cabada@usc.es; paula.cambeses.franco@usc.es}
\begin{document}
	\maketitle
	\begin{abstract}
In this paper, we analyze a second-order differential equation with a piecewise constant argument and reflection coupled to periodic boundary conditions. Our main contribution is the construction of the related Green's function and a detailed analysis of its properties. In particular, we determine the region in which the Green's function has constant sign, depending on the parameters $m$ and $M$ on which it depends. In some cases, we are able to characterize these parameter values in terms of the first eigenvalue related to suitable Dirichlet problems. Building in these results, we apply the Krasnosel'skii method to establish the existence of solutions for different nonlinear problems, and prove the existence of a positive solution of a perturbed Schrodinger equation.
		
	\end{abstract}



\section{Introduction}
\label{sec:1}

Differential equations with piecewise constant arguments and involution equations have attracted significant attention due to their diverse applications. Piecewise constant argument models are used to describe systems where variables and/or some of their derivatives are evaluated at discrete intervals. These models are particularly relevant in fields like control engineering and biomedicine. Involution equations, often involving reflection, also appear in biomedicine and physical phenomena, such as the heat equation for bent rods.

The study of involution equations date back to Silberstein's in $1940$, \cite{silberstein1940xvii}. Over time, these equations have been explored in numerous works \cite{aftabizadeh1988bounded,hendersonnontrivial,o1994existence,cabada2015differential}. Similarly, research on differential equations with piecewise constant arguments emerged in the $1980$s, blending methods from both differential and difference equations to address their challenges \cite{myshkis1977certain,torres2024oscillations}.

Our work combines the study of these two types of equations, drawing inspiration from the methodology of \cite{cacam} and employing Green's functions. The role of Green's functions is particularly significant: identifying regions where they maintain a constant sign enables the application of techniques such as the method of lower and upper solutions or fixed-point theorems. These tools, in turn, allow us to guarantee the existence of constant-sign solutions for a broad class of nonlinear boundary value problems involving involution and piecewise constant dependence in some of their variables.

Our primary objective is to derive the explicit form of the Green's function for a second-order problem with involution and piecewise constant arguments, analyze its main properties and apply comparison principles to approximate numerically the region where this function maintains a constant sign. In some cases, we also provide an analytical characterization of this constant-sign region in terms of eigenvalues of related Dirichlet problems, an approach that is novelty and particularly appealing because it does not require the explicit computation of the Green's function. Armed with this information, we will employ the Krasnosel'skii method to demonstrate the existence of solutions for a nonlinear problem. In particular, we ensure the existence of a positive solution of a perturbed Schrodinger equation.

This article is organized as follows: In Section \ref{motivation} we present the motivation for studying this problem, introducing the Schrodinger equation to be studied at the end of the paper. Section \ref{preli}, provides the preliminary results required for the subsequent analysis. In Section \ref{green}, we obtain the expression of the Green's function for a second-order problem with involution and piecewise constant arguments and examine some of its key properties. Section \ref{positivo} presents a numerical approximation of the region where the Green's function has constant sign, comparing it to analytical expectations. Section \ref{autovalores} offers a characterization of this constant-sign region based on the eigenvalues of related Dirichlet problems, which is particularly valuable as it avoids the need for explicitly computing the Green's function. Finally, in Section \ref{monotono}, we apply the Krasnosel'skii methods to prove the existence of solutions of suitable nonlinear problems, as, among others, the Schrodinger equation.

\section{Motivation and applications}
\label{motivation}

Equations of the form
\begin{equation*}
	v''(t)=f(t,v(t),v(-t),v([t]))
\end{equation*}
arise naturally in situations where a system exhibits continuous dynamics, discrete interventions and temporal symmetry or memory effects. These type of functional differential equations capture phenomena that cannot be modeled accurately by standard ordinary differential equations alone.

Building on this idea, it has been observed since the first experiments with Bose-Einstein condensates (BECs) in optical traps that the dynamics of quantum particles in periodic media cannot be fully described by the classical linear Schrödinger equation. The need to account for inter-particle interactions, periodic structures, and nonlocal effects has motivated the formulation of more complex versions of the Schrödinger equation.

In particular:

\begin{itemize}
	\item Spatial reflection, incorporated via terms like $\psi(-x,t)$, allows modeling symmetries and nonlocal effects in systems with $\mathcal{PT}$-symmetric potentials, which have been shown to play a crucial role in the formation and stability of solitons \cite{WenYan2017}.
	\item The dependence of the potential on the cell-averaged density, represented by $|\psi([x],t)|^2$, naturally arises when studying BECs in optical lattices or superlattices, where the local density of atoms modifies the effective potential in each cell \cite{MalomedDong2013}.
\end{itemize}

Combining these effects, our model is described by the nonlocal nonlinear Schrödinger equation:

\begin{equation}
i \hbar \, \frac{\partial \psi(x,t)}{\partial t} 
= -\frac{\hbar^2}{2m_p} \frac{\partial^2 \psi}{\partial x^2} 
+ \alpha \, |\psi([x],t)|^2 \psi(x,t) 
+ \beta \, \psi(-x,t) \quad x \in I:=[-T,T], \quad t \in [a,b],
\end{equation}
where $\hbar$ is the reduced Planck constant, $m_p$ the particle's mass and $\alpha$ and $\beta$ are real constants.

In the stationary case, assuming solutions of the form $\psi(x,t) = e^{-i \mu t/\hbar} \phi(x)$, the previous equation reduces to:

\begin{equation}
\mu \, \phi(x) = -\frac{\hbar^2}{2m_p} \frac{d^2 \phi(x)}{dx^2} 
+ \alpha |\phi([x])|^2 \phi(x) 
+ \beta \, \phi(-x), \quad x \in I,
\label{prorixi}
\end{equation}
where $\mu$ is the chemical potential or eigenenergy of the stationary state.

This work represents a first step in the study of this type of equations, combining cell-based nonlinear interactions and spatial reflection effects, and opens the door to the analysis of discrete solitons, breathers, and energy bands in periodic media with nonlocal symmetry. Such models may be useful for applications in optical superlattices, BECs in lattices, and simulations of complex quantum systems.

%
%

\section{Preliminaries}
\label{preli}
In this section, we collect notations, definitions, and results, which are needed in the sequel. 

We undertake a review of the concept of involution following references \cite{wiener1993generalized, wiener2002glimpse}.

\begin{definition}
Let $A \subset \mathbb{R}$ be a set containing more than one point, and let $f: A \rightarrow A$ be a function such that $f$ is not the identity $Id$. Then, $f$ is an involution if and only if
\begin{equation*}
f^{2} \equiv f \circ f=Id \textup{ on }A
\end{equation*}
or, equivalently, if
\begin{equation*}
f=f^{-1} \textup{ on }A.
\end{equation*}
If $A= \mathbb{R}$, we refer to $f$ as a strong involution.
\end{definition}

If the function $f$ represents a reflection, meaning $f(x)=-x$, then $f$ is a particular case of an involution at any symmetric interval centered at $0$.

Building on the theoretical framework outlined in \cite[Section 3]{cabada2015differential}, we can derive the Green's functions for differential equations with involution.

In addition, we will need to analyze differential problems that involve piecewise constant arguments, such as those of the form $v([t])$, where the function $[t]$ is defined as follows: 
\begin{equation*}
[t]=
\left\{
\begin{array}{lll}
n, & \mbox{\rm if}  & t \in [n,n+1) \\
-n, & \mbox{\rm if}  & t \in (-n-1,-n].
\end{array}
\right.
\end{equation*}
where $n \in \{0,1,2,\ldots\}$. Observe that $[t]=0$ for every $t \in (-1,1)$.

During this work, we will denote by $\Lambda$ the set of all functions $v:I \rightarrow \mathbb{R}$ that are continuous on $\hat{I}=[-T,[-T]) \cup [[-T],[-T]+1) \cup \ldots \cup [-2,-1 )\cup [-1,1) \cup [1,2) \cup \ldots \cup [[T]-1,[T]) \cup  [[T],T]$, and such that $v(t^{-}) \in \mathbb{R}$ exists for all $t \in \{[-T], \ldots, -1,1, \ldots, [T]\}$. Additionally, if $v \in \Lambda$, we understand that $v(t)=v(t^{+})$ for all $t \in \{[-T], \ldots, -1,1, \ldots, [T]\}$, $t^+ \in I$, assuming that such definition has sense.

\begin{remark}
	If $T \in \mathbb{N}$, then the point $[-T]$ and $[T]$ lie at the boundaries of $I$, and therefore the limits $v([-T]^-)$ and $v([T]^+)$ do not exist. In this case $v \in C(\hat{I} \cup \{T\})$.
\end{remark}

Furthermore, let $\Omega^{2}$ denote the set of all functions $v:I \rightarrow \mathbb{R}$ such that $v$ is $W^{2,1}(I)$ and there exists $v'' \in \Lambda$, where
\begin{equation*}
W^{2,1}(I)=\{v \in C^{1}(I), v'\in \mathcal{AC}(I)\},
\end{equation*}
\begin{equation*}
		AC(I) = \Big\{\, f:I\to\mathbb{R}\ \Big|\ 
		 \exists f' \in \mathcal{L}^1(I), \, f(t)-f(s)=\int_s^t{f'(r) \mathrm{d}r}, \, \forall \, t,s \in I
		\Big\},
\end{equation*}
being 
\begin{equation*}
	\mathcal{L}^{1}(I)=\{f \textup{ is a Lebesgue measurable function on $I$ such that }  \int_{I}{|f|}< \infty\}.
\end{equation*}

By following the procedure outlined in \cite[Section 3]{cacam}, we can derive the Green's function for any differential equation involving both involution and a piecewise constant argument of the form $v([t])$.

\section{Study of the equation $v''(t)+mv(-t)+Mv([t])=\sigma(t)$ with periodic conditions}
\label{green}
\label{subsec:2}
In this section, we focus on analyzing a linear second-order differential equation that includes a reflection term and a piecewise constant function. We will consider periodic boundary conditions. Specifically, we will address the following equation:
\begin{equation}
	\begin{aligned}
\label{tot1}
v''(t) + m v(-t) + M v([t]) &= \sigma(t), \quad  \text{a.e. } t \in I, \\
v(T) &= v(-T), \quad v'(-T) = v'(T),
	\end{aligned}
\end{equation}
where $T>0$, $m$, $M$ are real constants and $\sigma \in \mathcal{L}^{1}(I)$.

We will say that a function $v: I \rightarrow \mathbb{R}$ is a solution of Problem \eqref{tot1} if $v \in \Omega^{2}$ and satisfies equations \eqref{tot1}.


As a first step in the study of Problem \eqref{tot1}, we will obtain the expression of the Green's function for the particular case of $M=0$, following the results presented in \cite[Section 3]{cabada2015differential}. 

\subsection{Study of the equation $v''(t)+mv(-t)=\sigma(t)$ coupled to boundary periodic conditions}

We consider the following second-order problem with reflection, which is a particular case of Problem \eqref{tot1}:
\begin{equation}
\begin{aligned}
v''(t)+m\,v(-t)&=\sigma(t), \, \, \, \, \textup{ a.e. } t \in I,
\label{simp1} \\
v(T)&=v(-T), \quad v'(-T)=v'(T),
\end{aligned}
\end{equation}
where $T>0$, $m \in \mathbb{R}$ and $\sigma \in \mathcal{L}^{1}(I)$.

We denote by $G_{m}$ the Green's function related to Problem \eqref{simp1}. 
 Accordingly, by applying the algorithm described in  \cite{cabada2012computation} together with the formalism of \cite[Section 3.2.1]{cabada2013comparison}, we obtain the explicit expression of $G_{m}$ \cite{greensfunctionsreflection}.
 
Observing the expression, we note that $G_{m}(t,s)=G_m(-t,-s)$ and $G_{m}(t,s)=G_m(s,t)$ for all $s, t \in I$. Therefore, by writing the function only on the triangle $-t \leq s \leq t$, we can recover the full expression by symmetry. For simplicity, we will present it only in this region.

If $m>0$, and $-t \leq s \leq t$, we have that:
\begin{equation}
	G_{m}(t,s) = \frac{\cos (\sqrt{m} s) \csc (\sqrt{m} T) \cos (\sqrt{m} (t-T)) + \sinh (\sqrt{m} s) \text{csch}(\sqrt{m} T) \sinh (\sqrt{m} (t-T))}{2 \sqrt{m}}.
	\label{gexplicita}
\end{equation}

If $m<0$, and $-t \leq s \leq t$, then: 
\begin{equation}
	\label{expgnegativa}
	G_{m}(t,s) =
		\frac{\sin (\sqrt{-m} s) \csc (\sqrt{-m} T) \sin (\sqrt{-m} (t-T))-\cosh (\sqrt{-m} s) \text{csch}(\sqrt{-m} T)
			\cosh (\sqrt{-m} (t-T))}{2 \sqrt{-m}}.
\end{equation}

With this in hand, we now present a series of remarks about the properties of the Green's function $G_{m}$.
\begin{remark}
When $|m| \neq \left(k \pi /T\right)^2$, $k \in \mathbb{N} \cup \{0\}$, the Green's function $G_{m}$ is well-defined and the Problem \eqref{simp1} has a unique solution given by the expression
\begin{equation*}
v(t):=\int_{-T}^{T}{G_{m}(t,s) \sigma(s) \mathrm{d}s}.
\end{equation*}
Note that for $m=0$, the problem does not have a unique solution.
\end{remark}

\begin{remark}
For all $T>0$ and $|m| \neq \left(k \pi/T \right)^2$, $k \in \mathbb{N}\cup \{0\}$, is immediate to verify that $v_{0}(t)=\frac{1}{m}$ is the unique solution of Problem \eqref{simp1} for all $t \in I$ with $\sigma(t)=1$. 
 Moreover, it holds that
\begin{equation}
v_{0}(t)=\frac{1}{m}=\int_{-T}^{T}G_{m}(t,s) \mathrm{d}s, \quad \forall t \in I.
\label{vint}
\end{equation}
\end{remark}

On the other hand, considering the properties of Green's functions and the expression of $G_m$, we can deduce a series of properties of the function $G_{m}$.

\begin{proposition}
$G_{m}$ satisfies the following properties:
\begin{enumerate}
\item $G_{m}$ exists and it is continuous on $I \times I$,
\item $\frac{\partial G_{m}}{\partial{t}}$ and $\frac{\partial^{2}G_{m}}{\partial{t}^{2}}$ exist and are continuous on $\{(t,s) \in I \times I | s \neq t\}$,
\item $\frac{\partial G_{m}}{\partial{t}}(t,t^-)$ and $\frac{\partial G_{m}}{\partial{t}}(t,t^+)$ exist and are finite for all $t \in I$ and they satisfy $\frac{\partial G_{m}}{\partial{t}}(t,t^-)-\frac{\partial G_{m}}{\partial{t}}(t,t^+)=1$, for all $t \in I$,
\item $\frac{\partial^{2}G_{m}}{\partial{t}^{2}}(t,s)+m\,G_{m}(-t,s)=0$ for all $t,s \in I$, $s \neq t$,
\item $G_{m}(T,s)=G_{m}(-T,s)$ for all $s \in (-T,T)$,
\item $\frac{\partial{G_{m}}}{\partial{t}}(T,s)=\frac{\partial{G_{m}}}{\partial{t}}(-T,s)$ for all $s \in (-T,T)$,
\item $G_{m}(t,s)=G_{m}(-t,-s)$ for all $t,s \in I$,
\item $G_{m}(t,s)=G_{m}(s,t)$ for all $t,s \in I$,
\item $G_{m}(t,T)=G_{m}(t,-T)$ for all $t \in (-T,T)$,
\item $\frac{\partial{G_{m}}}{\partial{s}}(t,T)=\frac{\partial{G_{m}}}{\partial{s}}(t,-T)$ for all $t \in (-T,T)$,
\item $\frac{\partial^{2}G_{m}}{\partial{s}^{2}}(t,s)+m\,G_{m}(t,-s)=0$ for all $t,s \in I$, $s \neq t$.
\end{enumerate}
\label{propg}
\end{proposition}

Furthermore, from expressions \eqref{gexplicita}, \eqref{expgnegativa}, and \eqref{vint}, together with the previous properties, the sign of the Green's function $G_{m}$ can also be determined.

We begin by presenting the following remark:
\begin{remark}
	\label{necesario}
	From expression \eqref{vint}, we have that if $G_m$ has a constant sign on $I \times I$, then
	\begin{equation*}
		m \, G_{m}(t,s) \geq 0 \textup{ on }I \times I.
	\end{equation*}
\end{remark}

Taking all of the above into account, we arrive at the following Theorem.

\begin{theorem}
	\label{signog}
\begin{enumerate}
\item If $m \in \left(0, \left( \frac{\pi}{2T}\right)^2 \right)$, then $G_{m}$ is strictly positive on $I \times I$.
\item If $m=\left(\dfrac{\pi}{2T}\right)^2$, $G_{m}$ vanishes at $P:=\{(-T,-T),(0,0),(T,T),(T,-T),(-T,T)\}$ and is strictly positive on $(I \times I) \setminus P$.

\item Let $\overline{c}$ be the smallest positive solution of
\begin{equation}
	\tan({\overline{c}})=\frac{1}{\tanh{\left(\overline{c}\right)}}
\label{solc}
\end{equation}
which is approximately $\overline{c} \approx 0.937552$. Therefore, we have that, if $m \in \left(-\left(\frac{2\overline{c}}{T} \right)^2,0 \right)$, then function $G_m$ is strictly negative on $I \times I$.


\item If $m=-\left(\frac{2 \overline{c}}{T} \right)^2$, $G_m$ vanishes at $P_1:=\{(-T/2,T/2),(T/2,-T/2)\}$.

\item If $m \in \mathbb{R} \setminus \left[-\left( \frac{2 \overline{c}}{T}\right)^2,\left(\frac{\pi}{2T}\right)^2\right]$, then $G_{m}$ changes sign on $I \times I$.
\end{enumerate}
\end{theorem}

\begin{proof}
First, taking into account the symmetry of the function $G_{m}$ derived from properties $7$ and $8$ of Proposition \ref{propg}, to study the sign of the function $G_{m}$, we can restrict ourselves to analyzing the behavior of the function in the triangle $-s \leq t \leq s$, with $(t,s) \in I \times I$. 

On the one hand, if $G_{m}$ is strictly positive, then, from the previous remark, it is necessary that $m>0$. Moreover, from  properties $4$ and $11$ of Proposition \ref{propg}, we can guarantee the concavity of the functions $G_{m}(\cdot, s)$ and $G_{m}(t, \cdot)$. From this, we can deduce that the minimum of $G_{m}$, when it is positive, can occur only either at $t=s$ or $t=-s$.

Therefore, to analyze the region where the function is strictly positive, it is enough to study the cases where $t=s$ or $t=-s$ with  $-s \leq t \leq s$ and $m>0$. By analyzing these expressions and examining their minima, it can be easily verified that, while $0 \leq m\,T<\frac{\pi}{2}$, the function $G_{m}$ is positive, and that at $m\,T=\frac{\pi}{2}$, it vanishes at $(0,0)$, $(T,T)$, and $(-T,T)$. Taking into account the symmetry of $G_{m}$, we deduce the assertions $1$ and $2$.

On the other hand, to study the region where the Green's function is negative, it suffices, from previous remark, to consider the triangular domain defined by $-t \leq s \leq t$, with $(t,s) \in I \times I$ and $m<0$.

 For simplicity, we will use the notation $\alpha=\sqrt{-m}>0$. In this case:
\begin{equation}
	G_m(t,s)=	\frac{\sin (\alpha s) \csc (\alpha T) \sin (\alpha (t-T))-\cosh (\alpha s) \text{csch}(\alpha T)
		\cosh (\alpha (t-T))}{2 \alpha}.
		\label{expneg}
\end{equation}

Differentiating with respect to $t$, we obtain
\begin{equation}
	\frac{\partial}{\partial{t}}G_m(t,s)=\frac{1}{2}\left(\sin{(\alpha s)} \csc{(\alpha T)} \cos{(\alpha(t-T))}-\cosh{(\alpha s)}\textup{csch}(\alpha T) \sinh{(\alpha (t-T))} \right).
	\label{eqgt}
\end{equation}

As we are interested in finding a point $(\hat{t}, \hat{s}) \in I \times I$, where the maximum of $G_m$ is attained, we set the above expression equal to zero, thus obtaining

\begin{equation}
	\cosh{(\alpha \hat{s})}\textup{csch}(\alpha T)=\dfrac{\sin{(\alpha \hat{s})}\csc{(\alpha T)}\cos{(\alpha (\hat{t}-T))}}{\sinh{(\alpha (\hat{t}-T))}}.
	\label{igualadacero}
\end{equation}

Next, we substitute the above expression into \eqref{expneg}, obtaining
\begin{equation}
	G_m(\hat{t},\hat{s})=\frac{\sin{(\alpha \hat{s})}\csc{(\alpha T)}\sin{(\alpha (\hat{t}-T))}-\sin{(\alpha \hat{s})} \csc{(\alpha T)} \cos{(\alpha (\hat{t}-T))} \coth{(\alpha (\hat{t}-T))}}{2 \alpha}.
	\label{goxala}
\end{equation}

Since we are looking for the point where the maximum is zero, by setting the previous expression equal to zero we obtain
\begin{equation}
	\tan{(\alpha (\hat{t}-T))}=\dfrac{1}{\tanh{(\alpha (\hat{t}-T))}}.
	\label{igualadacero2}
\end{equation}
It is important to note that if we consider $\hat{s}=0$, the expression \eqref{goxala} also equals zero. However, it is obvious that, in such a case, equation \eqref{igualadacero} would not be satisfied.

Next, the same procedure is repeated with respect to $s$: we differentiate with respect to $s$, set the derivative to zero, substitute it into expression \eqref{expneg}, and then set it equal to zero again. In this case, we would obtain that
\begin{equation}
	\tan{(\alpha \hat{s})}=\dfrac{1}{\tanh{(\alpha \hat{s})}}.
	\label{exparas}
\end{equation}

As a consequence, from expressions \eqref{igualadacero2} and \eqref{exparas}, we must obtain the smallest positive value $\overline{c}$ that satisfies equations \eqref{solc}.
Moreover, we conclude that
\begin{equation}
	\overline{c}=|\overline{\alpha} \hat{s}|=-\overline{\alpha} (\hat{t}-T),
	\label{relcs}
\end{equation}
where $\overline{\alpha}>0$. 

Due to the old character of both functions: $\tan{(x)}$ and $\tanh{(x)}$, it is clear that $\overline{c}$ is a solution of \eqref{solc} if and only if $-\overline{c}$ is also a solution.

Moreover, from equation \eqref{relcs}, there are two possible cases: 
\begin{equation*}
\overline{c}= \overline{\alpha} \hat{s}>0 \textup{ with } \hat{s}>0 \textup{ or } \overline{c}=-\overline{\alpha} \hat{s}>0 \textup{ with } \hat{s}<0.
\end{equation*}

\begin{enumerate}
	\item Case 1: $\hat{s}>0$
	
Substituting $\overline{c}= \overline{\alpha} \hat{s}$ into equation \eqref{igualadacero}, we obtain
\begin{equation*}
	\frac{\sin{(\overline{\alpha} T)}}{\sinh{(\overline{\alpha} T)}}=\frac{\sin{(\overline{c})} \cos{(\overline{c})}}{-\cosh{(\overline{c})} \sinh{(\overline{c})}}= -\frac{\sin{(2 \overline{c})}}{\sinh{(2 \overline{c})}} <0.
\end{equation*}
Indeed, by equation \eqref{solc} we know that $0< 2 \overline{c}< \pi$, and, therefore, $\sin{2 \overline{c}}>0$. Hence, if a solution exists, it necessarily satisfies
 $$\overline{\alpha}T> \pi.$$
\item Case 2: $\hat{s}<0$

If $\hat{s}<0$, a similar computation using equation \eqref{igualadacero} yields
\begin{equation*}
	\frac{\sin{(\overline{\alpha} T)}}{\sinh{(\overline{\alpha} T)}}=\frac{\sin{(-\overline{c})} \cos{(-\overline{c})}}{\cosh{(-\overline{c})} \sinh{(-\overline{c})}}= \frac{\sin{(2 \overline{c})}}{\sinh{(2 \overline{c})}} >0.
\end{equation*}
In this case, since $2 \overline{c} \in (0, \pi)$, the function
$$f(x)= \frac{\sin{(2x)}}{\sinh{(2x)}}
$$
is injective on this interval. Consequently, there exists a unique $\overline{\alpha}>0$ such that
\begin{equation}
	\overline{\alpha} T \in (0, \pi) \textup{ and }\overline{\alpha} T=2 \overline{c}.
	\label{reltc}
\end{equation}
\end{enumerate}

Since we are interested in the smallest positive value of $\overline{\alpha}>0$, we must necessarily select this second case.

Furthermore, taking into account expressions \eqref{relcs} and \eqref{reltc}, we can calculate the values of $\hat{t}$ and $\hat{s}$, obtaining  $\hat{t}=T/2$ and $\hat{s}=-T/2$. 

Based on the preceding analysis, we can now assert that the maximum of the function, when it is negative, is attained at the point $(T/2,-T/2)$ and at its symmetric counterpart $(-T/2,T/2)$ (property $7$ of \ref{propg}). Thus, assertions $3$ and $4$ hold.

It is important to note that, although by the properties of the function $G_m$ given in Proposition \ref{propg} it is not differentiable at $t=s$, this does not affect our previous reasoning, since the jump of the first derivative of the Green's function, as stated in property $3$, already guarantees that the maximum cannot occur at $t=s$.

To prove assertion $5$ we will use expression $(4.4)$ and Proposition $4.1$ of \cite{cacam}. In this specific case, we have
\begin{equation*}
	G_{m_0}(t,s)=G_{m_1}(t,s)+(m_1-m_0)\int_{-T}^{T}{G_{m_1}(t,r)G_{m_0}(-r,s) \mathrm{d}r}
\end{equation*}
where $G_{m_0}$ and $G_{m_1}$ are the Green's functions related to Problem \eqref{simp1} corresponding to $m_0$ and $m_1 \in \mathbb{R} \setminus \{0\}$, respectively. 

From this, we deduce that, when $G_m$ has constant sign, it decreases as $m$ increases. Taking this into account, together with Remark \ref{necesario}, we conclude the proof of the Theorem.

\end{proof}

This is the first time, to the best of our knowledge, that the sign of the Green's function $G_m$ of Problem \eqref{simp1} is characterized.

With all this at hand, we  can now derive and characterize the explicit expression of the Green's function of Problem \eqref{tot1}.

\subsection{Green's Function for Problem \eqref{tot1}}

Once the explicit expression for the Green's function of the problem without a piecewise constant argument has been calculated and analyzed, we proceed to obtain and study the explicit expression of the Green's function for Problem \eqref{tot1}, which we denote by $H_{m,M}$. To do this, we follow the steps outlined in \cite[Section 3]{cacam}.

Let us assume that Problem \eqref{tot1} has a unique solution $v(t)$. Then, we can write it as
\begin{equation*}
v(t)=\int_{-T}^{T}{H_{m,M}(t,s)\sigma(s) \mathrm{d}s}.
\end{equation*}

Following equation $(3.5)$ in \cite{cacam}, we can express the Green's function $H_{m,M}$ in terms of $G_{m}$ arriving at the following expression: 
\begin{equation}
H_{m,M}(t,s)=G_{m}(t,s)-M\left[\sum_{j=[-T]}^{[T]}\sum_{i=[-T]}^{[T]}G_{m}(j,s)\int_{-T}^{T}G_{m}(t,r)\alpha_{ij}(r) \mathrm{d}r\right],
\end{equation}
where $G_{m}$ is the Green's function for the Problem \eqref{simp1} and $\alpha_{ij}: I \rightarrow \mathbb{R}$ are defined by
\begin{eqnarray*}
\alpha_{[-T]j}(r)&:=&\tilde{a}_{[-T]j}\chi_{(-T,-[T])}(r), \\ 
\alpha_{[-T+1]j}(r)& := &\tilde{a}_{[-T+1]j}\chi_{(-[T],-[T+1])}(r), \\
&\vdots& \\ 
\alpha_{0j}(r) &:=& \tilde{a}_{0j}\chi_{(\max\{-1,T\},\min\{1,T\}\})}(r), \\
&\vdots& \\
\alpha_{[T]j}(r) &:=& \tilde{a}_{[T]j}\chi_{([T],T)}(r),
\end{eqnarray*}
where $j \in \{[-T],[-T+1], \ldots, 0, \ldots, [T]\}$ and $\tilde{a}_{ij}$ are the elements of the inverse of the matrix $A$ given by
\begin{equation*}
A \equiv
\begin{pmatrix}
    Ma_{[-T][-T]}+1 & Ma_{[-T][-T+1]} & \cdots & Ma_{[-T]0} & \cdots & Ma_{[-T][T]} \\
    Ma_{[-T+1][-T]} & Ma_{[-T+1][-T+1]}+1 & \cdots & Ma_{[-T+1]0} &  \cdots & Ma_{[-T+1][T]} \\
    \vdots  & \vdots & \cdots & \vdots & \ldots & \vdots \\
    \vdots  & \vdots & \cdots & \vdots & \ldots & \vdots \\
    Ma_{[T][-T]} & Ma_{[T][-T+1]} & \cdots & Ma_{[T]0} & \cdots & Ma_{[T][T]}+1
\end{pmatrix},
\label{matriz}
\end{equation*}
where
\begin{align*}
 a_{l,[-T]} &:= \int_{-T}^{[-T]} G_{m}(l,s) \, \mathrm{d}s,\\ 
 a_{l,[T]} &:= \int_{[T]}^{T} G_{m}(l,s) \, \mathrm{d}s, \\
 a_{l,0} &:= \int_{\max\{-T,-1\}\}}^{\min\{1,T\}\}} G_{m}(l,s) \, \mathrm{d}s \\
 \intertext{ and }
a_{l,k} &:= \int_{k-1}^{k} G_{m}(l,s) \mathrm{d}s \textup{ for }k \in \{[-T]+1, \ldots , [T]-1\} \setminus \{0\}.
 \end{align*}

In the particular case where $T \in (0,1]$ we arrive at the following expression
\begin{equation}
H_{m,M}(t,s)=G_{m}(t,s)-M\frac{\int_{-T}^{T}G_{m}(t,r) \mathrm{d}r}{1+M \int_{-T}^{T}G_{m}(0,r) \mathrm{d}r}G_{m}(0,s).
\label{htpequeño}
\end{equation}
If we take into account expression \eqref{vint}, we finally obtain the expression
\begin{equation}
H_{m,M}(t,s)=G_{m}(t,s)-\frac{M}{m+M}G_{m}(0,s) \quad \forall (t,s) \in I \times I, \, \forall \, T \in (0,1].
\label{casosencillo}
\end{equation}

From now on, as in \cite[Section 3]{cacam} we will introduce the following sets:
\begin{equation*}
D \equiv \{-[T], \ldots, 0, \ldots, [T]\}
\end{equation*}
and
\begin{equation*}
D_{t} \equiv D \cup \{t,-t\}, \textup{ for }t \in I \textup{ given}.
\end{equation*}
Using the definition of the Green's function and the properties of $G_{m}$ shown in Proposition \ref{propg}, it is not difficult to verify that:

\begin{proposition}
\label{carah}
The Green's function $H_{m,M}$, related to Problem \eqref{tot1}, satisfies the following conditions:
\begin{enumerate}
\item $H_{m,M}(\cdot,s)$ is well-defined and continuous for all $t \in I$, of class $C^{1}$ for all $t \in I$, $t \neq s$, and of class $C^{2}$ for all $t \in I$, $t \neq s$ and $t \notin D \setminus \{0\}$.
\item $H_{m,M}(t, \cdot)$ is well-defined and continuous for all $s \in I$. Moreover, $H_{m,M}(t,\cdot)$ is of class $C^{2}$ for all $s \in I$, $s \notin D_{t}$.
\item For each $s \in (-T,T)$, the function $t \rightarrow H_{m,M}(t,s)$ is the solution of the following differential equation:
\begin{equation}
\frac{\partial^{2}}{\partial{t}^{2}}H_{m,M}(t,s)+m\,H_{m,M}(-t,s)+M\,H_{m,M}([t],s)=0,
\end{equation}
for all $t \in I$, $t \neq s$ and $t \notin D \setminus \{0\}$.
\item For each $t \in (-T,T)$, $t \notin D$, the lateral limits:
\begin{equation*}
\frac{\partial}{\partial{t}}H_{m,M}(t^-,t)=\frac{\partial}{\partial{t}}H_{m,M}(t,t^+) \quad \textup{and} \quad \frac{\partial}{\partial{t}}H_{m,M}(t,t^-)=\frac{\partial}{\partial{t}}H_{m,M}(t^+,t),
\end{equation*}
exist and are real. Furthermore,
\begin{equation*}
\frac{\partial}{\partial{t}}H_{m,M}(t^+,t)-\frac{\partial}{\partial{t}}H_{m,M}(t^-,t)=\frac{\partial}{\partial{t}}H_{m,M}(t,t^-)-\frac{\partial}{\partial{t}}H_{m,M}(t,t^+)=1.
\end{equation*}
\item For each $s \in (-T,T)$, the function $t \rightarrow H_{m,M}(t,s)$ satisfies the boundary conditions:
\begin{equation*}
H_{m,M}(T,s)=H_{m,M}(-T,s) \quad \textup{and} \quad \frac{\partial}{\partial{t}}H_{m,M}(T,s)=\frac{\partial}{\partial{t}}H_{m,M}(-T,s).
\end{equation*}
\item For each $t \in (-T,T)$, the function $s \rightarrow H_{m,M}(t,s)$ satisfies the boundary conditions:
\begin{equation*}
	H_{m,M}(t,T)=H_{m,M}(t,-T) \quad \textup{and} \quad \frac{\partial}{\partial{s}}H_{m,M}(t,T)=\frac{\partial}{\partial{s}}H_{m,M}(t,-T).
\end{equation*}
\item $H_{m,M}(t,s)=H_{m,M}(-t,-s)$ for all $t,s \in I$.
\item $\frac{\partial^{2}}{\partial{s}^{2}}H_{m,M}(t,s)+m\,H_{m,M}(t,-s)=0$ for all $t, s \in I$, $s \notin D_{t}$.
\end{enumerate}
\end{proposition}

\section{Region of constant sign for the Green's function $H_{m,M}$}
\label{positivo}
The aim of this section is to determine, as a function of the parameters $m$ and $M$, the region where the Green's function $H_{m,M}$ maintains a constant sign on $I \times I$.

We begin by deriving a necessary condition for $H_{m,M}$ to be either positive or negative. To this end, we present the following result:

\begin{proposition}
A necessary condition for the Green's function $H_{m,M}$, related to Problem \eqref{tot1}, to be positive on $I \times I$ is that $m+M>0$. Similarly, a necessary condition for $H_{m,M}$ to be negative on $I \times I$ is that $m+M<0$.
\end{proposition}

\begin{proof}
We assume that Problem \eqref{tot1} has a unique solution for any $\sigma \in \mathcal{L}(I)$. The proof is a direct consequence of the fact that if $\sigma(t)=1$ for all $t \in I$ then $v(t)=\frac{1}{m+M}$ is its unique solution.
\end{proof}

\begin{remark}
Notice that the curve $M=-m$ is a curve of eigenvalues of problem \eqref{tot1}.
\end{remark}

Secondly, we will attempt to constrain the points $(t,s) \in I \times I$ where the function $H_{m,M}$ may attain its minimum when it is positive, or its maximum when it is negative, and approximate the region.

In this context, we present the following result:

\begin{proposition}
	\label{posicion}
	Let $H_{m,M}$ be the Green's function of Problem \eqref{tot1}, and suppose that it is of constant sign. Then:
	\begin{enumerate}
		\item If $m \geq 0$ and $M \geq 0$, the minimum, when $H_{m,M}$ is positive, must be attained at the diagonal $t=s$ for $s \in I$.
		\item If $m \geq 0$, the minimum, when $H_{m,M}$ is positive, or the maximum, when it is negative, must be attained either at integer values of $s$ with $s \in I \cap \mathbb{Z}$ or at the diagonal $s=t$ for $t \in I$.
	\end{enumerate}
\end{proposition}

\begin{proof}
	The first part follows directly from items (1), (3), (4) and (5) of Proposition \ref{carah}, which respectively address the regularity of $H_{m,M}(\cdot,s)$, its concavity under the assumptions of the hypotheses, the jump condition for the derivative at the diagonal, and the periodic boundary conditions.
	
	The second statement follows from items (2), (4), (6), and (8), which respectively describe the regularity of $H_{m,M}(t,\cdot)$, the jump condition at the diagonal, the periodic boundary conditions, and the fact that $H_{m,M}$ is concave when it is positive and convex when it is negative.
\end{proof}

On the other hand, due to the symmetry of the function $H_{m,M}$ shown in part $7$ of Proposition \ref{carah}, to search for the minimum or the maximum and delimit the region of constant sign, it will only be necessary to work with values of $t \in I$ and $s \in [0,T]$.

Following the steps outlined in \cite[Section 4.1]{cacam}, we can derive a relation between the Green's functions $H_{m,M_{0}}$ and $H_{m,M_{1}}$ related  to problems \eqref{tot1}, with parameters $(m,M_{0})$ and $(m,M_{1})$, respectively. The formula is the next one:
\begin{equation}
H_{m,M_{0}}(t,s)=H_{m,M_{1}}(t,s)+(M_{1}-M_{0})\int_{-T}^{T}{H_{m,M_{1}}(t,r)H_{m,M_{0}}([r],s) \mathrm{d}r}.
\label{relacion}
\end{equation}

From equation \eqref{relacion} and following steps in \cite{cacam}, it can be deduced that, since we are interested in finding $M_{0}$ such that $\min_{(t,s) \in I \times I}{H_{m,M_{0}}}(t,s)=H_{m,M_{0}}(\hat{t},\hat{s})=0$ for some $(\hat{t},\hat{s}) \in I \times I$, our objective is to find, for each $m$, the smallest $M_{0}$ satisfying
\begin{equation*}
M_{0}= \overline{T}_{m}(M_{0},\hat{t},\hat{s}) \textup{ for some } (\hat{t},\hat{s}) \in I \times I,
\end{equation*}
where the operator $\overline{T}_{m}: \mathbb{R} \times I \times I \rightarrow \mathbb{R}$ is defined as:
\begin{equation*}
\overline{T}_{m}(M_{0},t,s)=\frac{G_{m}(t,s)}{\int_{-T}^{T}{G_{m}(t,r)H_{m,M_{0}}([r],s)\mathrm{d}r}}.
\end{equation*}

Finding the minimum or maximum of $T_m(M_0,t,s)$ for $(t,s) \in I \times I$ can be significantly more efficient computationally than minimizing $H_{m,M_0}$. This is because the integral can be split into separate intervals, allowing $H_{m,M_0}$ to be evaluated only at integer values of $t \in I$ rather than for every value of $t$ within $I$. Furthermore, evaluating $H_{m,M_0}$ for large values of $T$ can be computationally expensive.


Taking this into account and following \cite[Section 4.2 and Subsection 5.4.2]{cacam}, we can numerically approximate the region where $H_{m,M}$ is positive for any $T$ given. For example, when $T=1.6$, we obtain the graph in Figure \ref{aproxregion2}.
\vspace{-0.5cm}
\begin{figure}[H]
	\centering
		\includegraphics[width=0.8\textwidth]{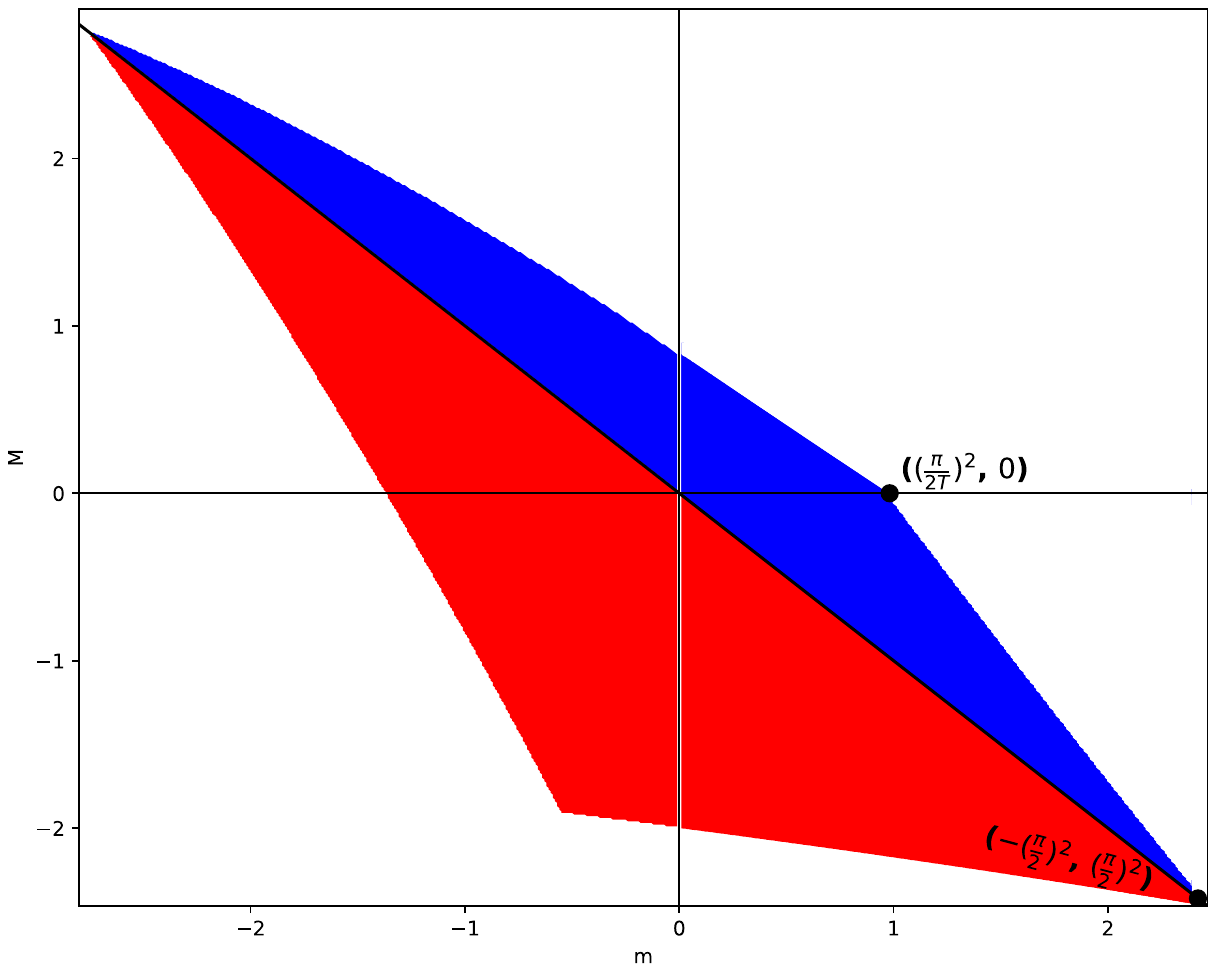}
		\vspace{-0.5cm}
	\caption{Numerical approximation of the regions in which the Green's function $H_{m,M}$ maintains a constant sign. The domain where $H_{m,M}>0$ is depicted in blue, while the domain where $H_{m,M}<0$ is depicted in red.}
	\label{aproxregion2}
\end{figure}

Once we have obtained a numerical approximation of the region of constant sign, we attempt to determine the exact expression of this region. To do so, we need to identify the points $(\hat{t},\hat{s}) \in I \times I$ where the Green's function $H_{m,M}$ attains either its minimum or maximum.

In the case when $m>0$, and taking Proposition \ref{posicion} into account, our numerical explorations suggest that the minimum of $H_{m,M}$ for the largest value of $M$ for which the function remains positive appears to be attained at the point $(T,T) \in I \times I$ when $m \in \left(0,\left(\frac{\pi}{2T} \right)^{2} \right)$, and at $(0,0) \in I \times I$ when $m \in \left(\left(\frac{\pi}{2T} \right)^2,\left(\frac{\pi}{2} \right)^2 \right)$. 

Similarly, when $H_{m,M}$ becomes negative and $m>0$, our computations indicate that its maximum is attained at $(T,0)$ (see Appendix \ref{apend}). 

On the other hand, when $m<0$, the numerical evidence suggests that the minimum of the Green's function $H_{m,M}$, when it is positive, can be attained either at $(0,0) \in I \times I$ or at $(3T/4,3T/4) \in I \times I$, whereas the maximum, when it is negative, appears to occur either at $(T,0) \in I \times I$ or at $(T/2,-T/2) \in I \times I$.

Therefore, by evaluating $H_{m,M}$ at these candidate points, identified through numerical inspection, setting the resulting expressions equal to zero, and solving for $M$ as a function of $m$, we obtain explicit formulas for the regions in which $H_{m,M}$ is positive or negative. This procedure yields the symbolic description of the region.

This analysis leads to the region displayed in Figure \ref{region16bo} for the case $T=1.6$.

\begin{figure}[H]
	\centering
	\includegraphics[width=0.6\textwidth]{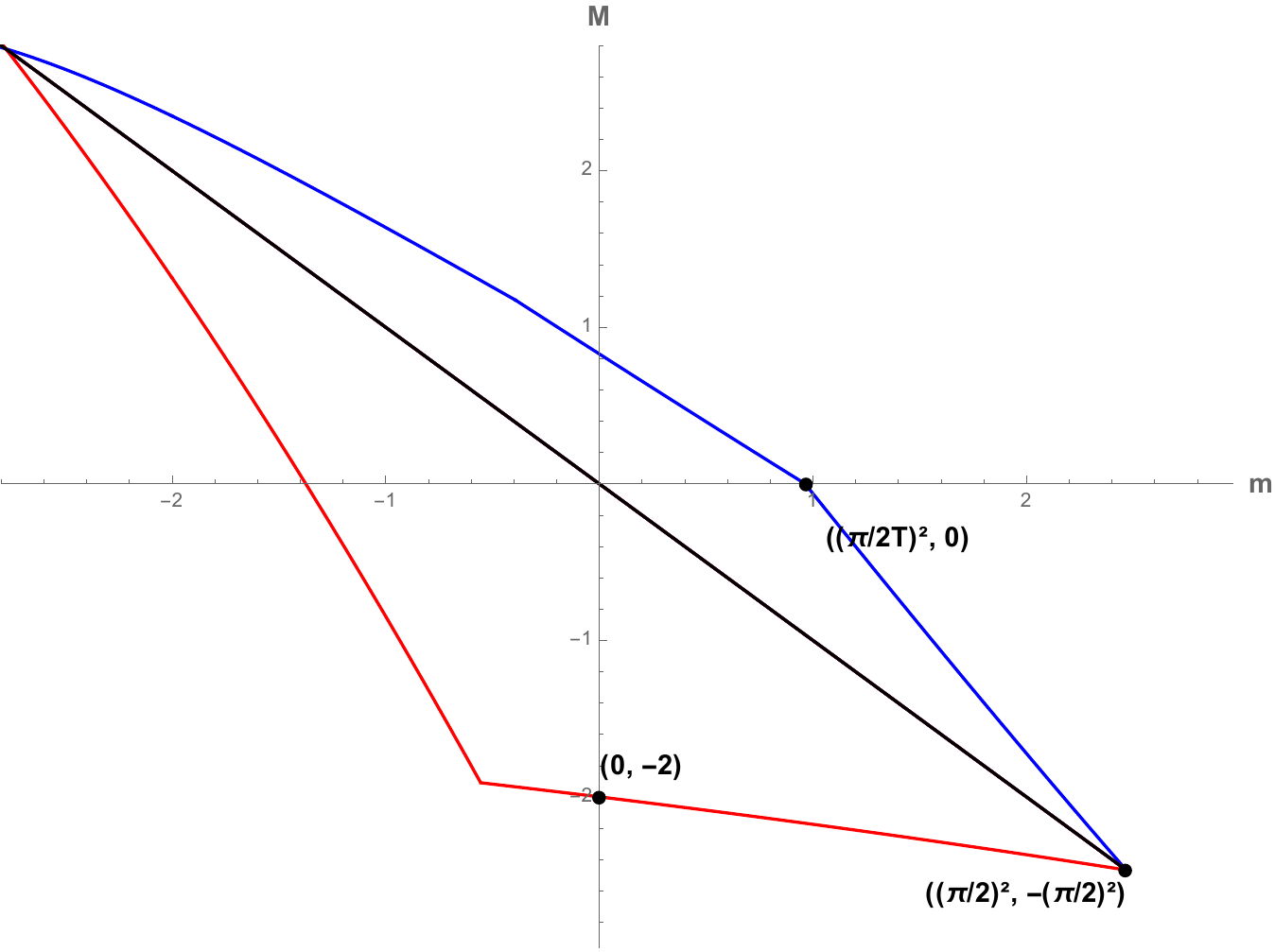}
	\caption{Explicitly conjectured regions in which the Green's function $H_{m,M}$ maintains a constant sign when $T=1.6$. The domain where $H_{m,M}>0$ is depicted between the black and blue curves, while the domain where $H_{m,M}<0$ is shown between the black and red curves.}
	\label{region16bo}
\end{figure}

Visually, we verify that the two regions shown in figures \ref{aproxregion2} and \ref{region16bo} coincide, providing additional evidence supporting the validity of our results.

We can apply the same reasoning to determine the regions in which the Green's function $H_{m,M}$ has a constant sign for other values of $T$.

Moreover, when $T \leq 1$, expression \eqref{htpequeño} provides an explicit formula for $H_{m,M}$, allowing us to work directly with it. Following the approach of Theorem \ref{signog}, that is, computing the partial derivatives of $H_{m,M}$ with respect to $t$ and $s$ and setting them to zero in each case $m \geq 0$ or $m<0$, or simply examining the diagonal $t=s$, one can determine symbolically the points where $H_{m,M}$ attains its minima or maxima, without relying on numerical approximations. This, in turn, makes it possible to obtain with full precision the region where $H_{m,M}$ has constant sign when $T \leq 1$.

For $T<1$, we can easily obtain the explicit expression for the regions of constant sign, taking into account the explicit form of the function $H_{m,M}$ given in equation \eqref{casosencillo}. The function $H_{m,M}$ is positive in the following cases:

\begin{equation*}
-m < M < 
\begin{cases}
	\dfrac{m}{-1 + \sec(\sqrt{m}\, T)}, & m \in \left(0, \left(\frac{\pi}{2T}\right)^2\right), \\[1em]
	\dfrac{2}{T^2}, & m = 0, \\[0.5em]
	\dfrac{m \cosh(\sqrt{-m}\, T)}{1-\cosh(\sqrt{-m}\, T)}, & m \in (\frac{\alpha_2}{T^2},0), \\[0.5em]
	F(m,T), & m \in \left(-\left(\frac{\pi}{T}\right)^2, \frac{\alpha_2}{T^2}\right),
\end{cases}
\end{equation*}
where $F(m,T)$ is given by
%
%
\[
- m - 
\frac{
	m \, \textup{csch}\!\left(\frac{\sqrt{-m} \, T}{4}\right) \, \textup{sech}\!\left(\frac{\sqrt{-m} \, T}{2}\right)
}{
	\coth\!\left(\frac{\sqrt{-m} \, T}{4}\right) - \textup{csch}\!\left(\frac{\sqrt{-m} \, T}{4}\right) \, \textup{sech}\!\left(\frac{\sqrt{-m} \, T}{2}\right) + 
	\tan\!\left(\frac{\sqrt{-m} \, T}{4}\right) + \tan\!\left(\frac{\sqrt{-m} \, T}{2}\right) + \tanh\!\left(\frac{\sqrt{-m} \, T}{2}\right)
}
\]
and $\alpha_2<0$ is the first root of equation
\begin{equation*}
	F\!\left(\alpha_2/T^2, T\right)
	=
	\dfrac{\alpha_2/T^2\cosh\!\left(\sqrt{-\alpha_2}\right)}
	{1-\cosh\!\left(\sqrt{-\alpha_2}\right)}.
\end{equation*}
This root does not depend on the parameter $T$ and is approximately equal to $\alpha_2 \approx -2.091$.

Analogously, the function $H_{m,M}$ is negative in the following cases:
 
 \begin{equation*}
 	-m < M < 
 	\begin{cases}
 		\dfrac{m}{-1+\cos{(\sqrt{m}T)}}, & m \in \left(0, \left(\frac{\pi}{2T}\right)^2\right), \\[1em]
 		-\dfrac{2}{T^2}, & m = 0, \\[0.5em]
 		\dfrac{m}{\cosh{(\sqrt{-m}T)}-1}, & m \in \left(\frac{\alpha_3}{T^2},0\right), \\[0.5em]
 		\frac{m\!\left(\coth\!\big(\tfrac{\sqrt{-m}\,T}{2}\big)-\tan\!\big(\tfrac{\sqrt{-m}\,T}{2}\big)\right)}
 		{-\coth\!\big(\tfrac{\sqrt{-m}\,T}{2}\big)+\textup{csch}\!\big(\tfrac{\sqrt{-m}\,T}{2}\big)+\tan\!\big(\tfrac{\sqrt{-m}\,T}{2}\big)}, & m \in \left(-\left(\frac{\pi}{T}\right)^2, \frac{\alpha_3}{T^2}\right),
 	\end{cases}
 \end{equation*}
 where $\alpha_3<0$ is the first root of the following equation
\begin{equation*}
	\frac{1}{\cosh(\sqrt{-\alpha_3})-1}
	=
	\frac{1\!\left(
		\coth\!\big(\tfrac{\sqrt{-\alpha_3}}{2}\big)
		-
		\tan\!\big(\tfrac{\sqrt{-\alpha_3}}{2}\big)
		\right)}
	{
		-\coth\!\big(\tfrac{\sqrt{-\alpha_3}}{2}\big)
		+
		\textup{csch}\!\big(\tfrac{\sqrt{-\alpha_3}}{2}\big)
		+
		\tan\!\big(\tfrac{\sqrt{-\alpha_3}}{2}\big)
	},
\end{equation*}
from which we deduce that $\alpha_3 \approx -2.693$.
 
 
We represent the region for $T=0.5$ and $T=1$ on Figure \ref{regionvarias}.
\begin{figure}[H]
	\centering
	\includegraphics[width=0.8\textwidth]{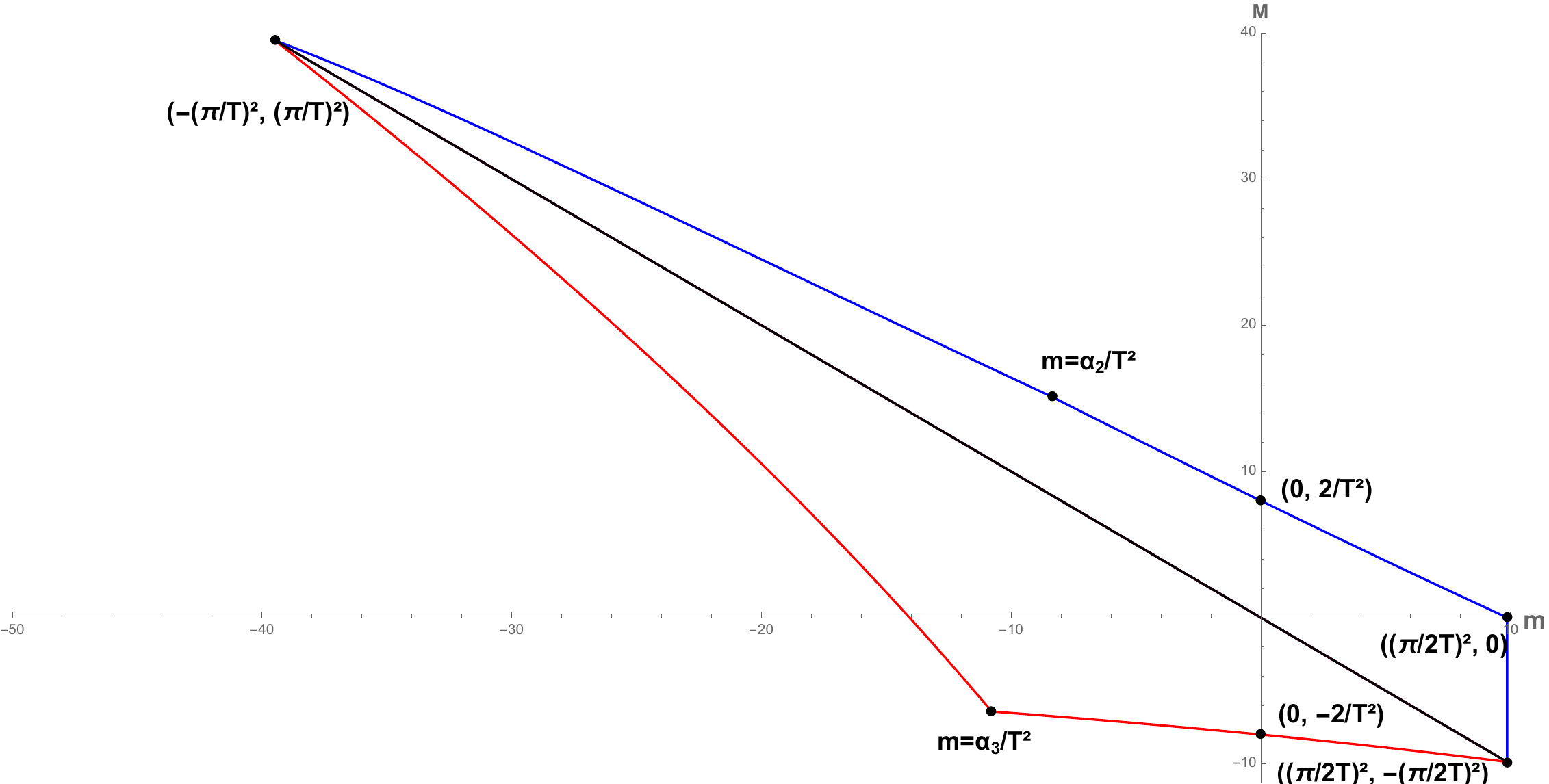} 
	\\
	\textbf{(a)} Regions of constant sign when $T=0.5$.
	\vspace{0.5cm}
	
	\includegraphics[width=0.8\textwidth]{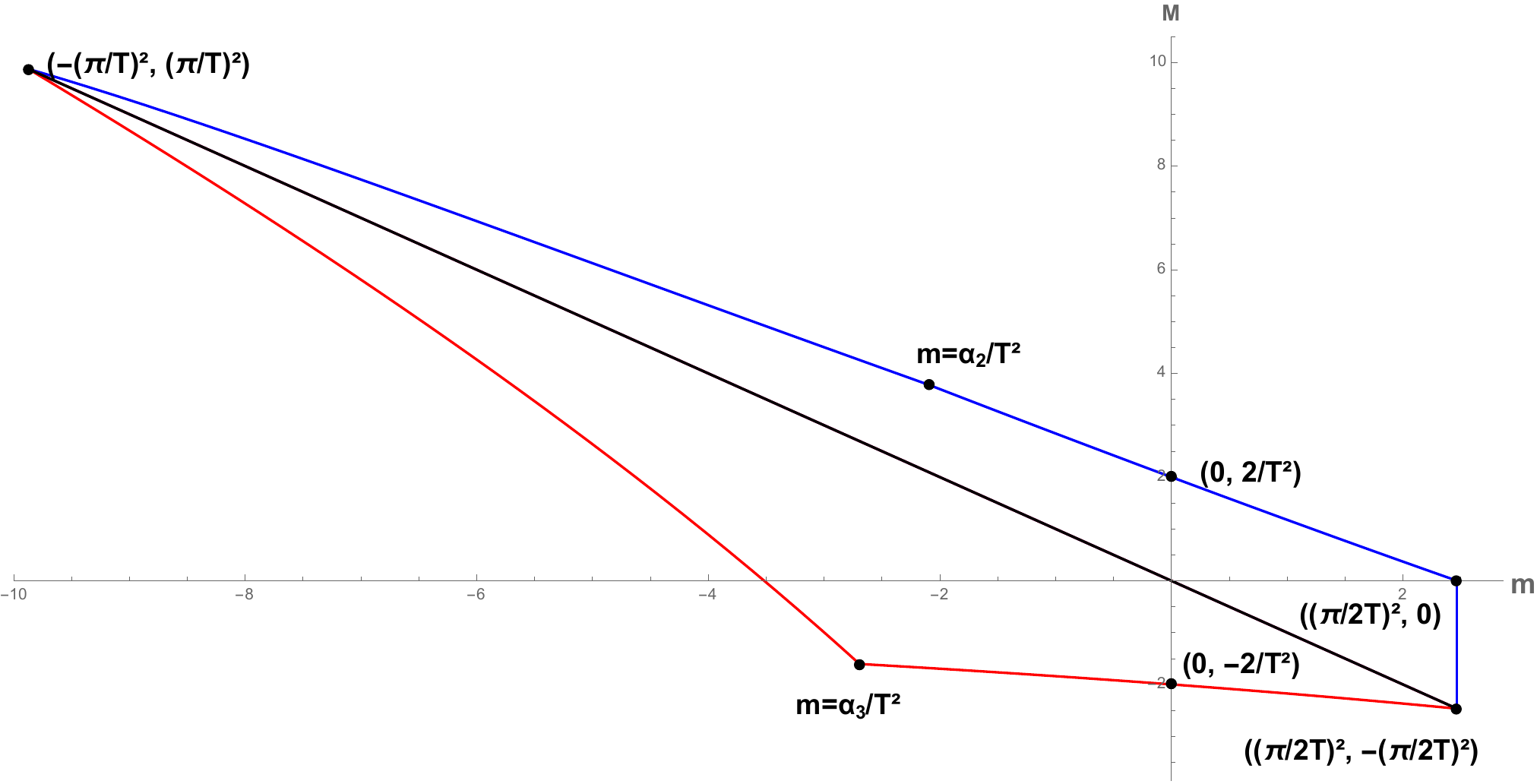} 
	\\
	\textbf{(b)} Regions of constant sign when $T=1$.
	\vspace{0.2cm}
	
	\caption{Representation of the regions of constant sign for the function $H_{m,M}$ for different values of $T$. The region with a positive sign is the one between the blue and black lines, while the region with a negative sign is the one between the black and red lines.}
	\label{regionvarias}
\end{figure}

For $T>1$, it is more difficult to directly obtain the explicit expression for the regions of constant sign. The function $H_{m,M}$ appears to be negative whenever $m \in \left(0, \left(\frac{\pi}{2}\right)^2 \right)$ and $\frac{m}{-1+\cos{(\sqrt{m})}}<M<-m$, or when $m=0$ and $-2<M<0$, regardless of the value of $T$. But such assertion remains as a conjecture.

\section{Eigenvalue characterization for the constant sign Green's function of $H_{m,M}$ when $m \geq 0$ and $M \geq 0$}
\label{autovalores}

In the case where $m \geq 0$ and $M \geq 0$, Proposition \ref{posicion} shows that the minimum of the Green's function $H_{m,M}$ when it is positive must be attained at $t=s$ with $s \in I$. This leads to the following Theorem:
\begin{theorem}
	\label{teoautovalor}
	Let $H_{m,M}$ be the Green's function of Problem \eqref{tot1} and suppose that $m \geq 0$, $M \geq 0$ and $m+M > 0$. Then $H_{m,M}$ is positive if, and only if, $M \in (-m, \lambda_{1})$, where $\lambda_{1}$ is the smallest positive Dirichlet eigenvalue of the problems:
	\begin{equation}
		\begin{aligned}
		z''(t)&=
		\left\{
		\begin{array}{lll}
			-m\, z(-t-2s_0)-M\, z(T+[T+t+s_0]-s_0), & \mbox{\rm if}  & t \in I_{s_0}^{1}, \\
			-m\, z(-t-2s_0)-M\, z(-T+[T+t+s_0]-s_0), & \mbox{\rm if}  & t \in I_{s_0}^{2}, \\
			-m\, z(-t-2s_0)-M\, z(T+[-T+t+s_0]-s_0), & \mbox{\rm if}  & t \in [-s_0,T-2s_0), \\
			-m \,z(-t-2s_0+2T)-M\, z(T+[-T+t+s_0]-s_0), & \mbox{\rm if}  & t \in [T-2s_0,T], \\
		\end{array} 
		\right. \\
		\label{znoentero}
		z(-T)&=z(T)=0,
		\end{aligned}
	\end{equation}
	with $I_{s_0}^1=[-T,\min\{[s_0]+1-s_0-T,-s_0\}\})$ and $I_{s_0}^{2}=[\min\{[s_0]+1-s_0-T\},-s_0)$, if $s_0 \notin \mathbb{N}$. Or, if $s_0 \in \mathbb{N}$,
	\begin{equation}
		\begin{aligned}
	z''(t)&=
	\left\{
	\begin{array}{lll}
		-m\, z(-t-2s_0)-M\, z(-T+[T+t+s_0]-s_0), & \mbox{\rm if}  & t \in [-T,-s_0), \\
		-m \,z(-t-2s_0)-M\, z(T+[-T+t+s_0]-s_0), & \mbox{\rm if}  & t \in [-s_0,T-2s_0), \\
		-m\, z(-t-2s_0+2T)-M\, z(T+[-T+t+s_0]-s_0), & \mbox{\rm if}  & t \in [T-2s_0,T], \\
	\end{array}
	\right. \\
	\label{zentero}
	z(-T)&=z(T)=0.
		\end{aligned}
	\end{equation}
	with $t \in I$ and $s_0 \in [0,T]$.
\end{theorem}

\begin{proof}
	From Proposition \ref{posicion}, it follows that the minimum of the Green's function $H_{m,M}$, when it is positive, is attained at $t=s$ with $s \in I$. However, due to the symmetry of $H_{m,M}$ described in item 7 of Proposition \ref{carah}, the search for the minimum can be restricted to the diagonal $t=s$ with $s \in [0,T]$.
	
	Fixing a value $s_{0} \in [0,T]$, the function $H_{m,M}(t,s_0)$, considered as a function of $t$, is a solution of \eqref{tot1} on the intervals $[-T,s_0)$ and $(s_{0},T]$, and satisfies Property 5 of Proposition \ref{carah}.
	
	Therefore, we can construct
	\begin{equation*}
	w(t)=
	\left\{
	\begin{array}{lll}
		v(t), & \mbox{\rm if}  & t \in [s_{0},T], \\
		v(t-2T), & \mbox{\rm if}  & t \in [T,2T+s_{0}]. 
	\end{array}
	\right.
	\end{equation*}
	
	This function is of class $C^{1}(s_0,2\,T+s_0)$ and, as previously shown, attains its minimum at $s_0$ and at $2\,T+s_0$, where it necessarily must take the value $0$. Moreover, if $s_0 \notin \mathbb{N}$, it must satisfy the following equation:
\begin{equation*}
	\begin{aligned}
	w''(t)&=
	\left\{
	\begin{array}{lll}
		-m\, w(-t+2T)-M \,w([t]+2T), & \mbox{\rm if}  & t \in [s_{0},\min\{[s_{0}]+1,T\}\}), \\
		-m\, w(-t+2T)-M\, w([t]), & \mbox{\rm if}  & t \in [\min\{[s_{0}]+1,T\},T), \\
		-m\, w(-t+2T)-M \,w([t-2T]+2T), & \mbox{\rm if}  & t \in [T,2T-s_{0}), \\
		-m\, w(4T-t)-M \,w([t-2T]+2T), & \mbox{\rm if}  & t \in [2T-s_{0},2T+s_{0}], \\
	\end{array}
	\right. \\
	w(s_0)&=w(2T+s_0)=0.
	\end{aligned}
\end{equation*}

In the case where $s_0 \in \mathbb{N}$, we have that
\begin{equation*}
	\begin{aligned}
	w''(t)&=
	\left\{
	\begin{array}{lll}
		-m\, w(-t+2T)-M\, w([t]), & \mbox{\rm if}  & t \in [s_0,T), \\
		-m\, w(-t+2T)-M\, w([t-2T]+2T), & \mbox{\rm if}  & t \in [T,2T-s_{0}), \\
		-m\, w(4T-t)-M\, w([t-2T]+2T), & \mbox{\rm if}  & t \in [2T-s_{0},2T+s_{0}], \\
	\end{array}
	\right. \\
	w(s_0)&=w(2T+s_0)=0.
	\end{aligned}
\end{equation*}

Next, we perform the translation
\begin{equation*}
	z(t)=w(T+t+s_0).
\end{equation*}

 It is not difficult to verify that function $z$ is of class $C^{1}(I)$ and satisfies equation \eqref{znoentero} if $s_{0} \notin \mathbb{N}$ and equation \eqref{zentero} if $s_{0} \in \mathbb{N}$. Moreover, the minimum of $z$ is attained at $-T$ and at $T$.
 
Therefore, the smallest value of $M(m)$ for which the function $H_{m,M}$ ceases to be positive must satisfy $z(-T)=z(T)=0$ and it coincides with the smallest eigenvalue of the Dirichlet problems \eqref{znoentero} if $s_{0} \notin \mathbb{N}$ or \eqref{zentero} if $s_{0} \in \mathbb{N}$.

\end{proof}

It is useful to analyze separately the particular cases in which either $m=0$ or $M=0$.

In the case $M=0$, the following result holds:

\begin{theorem}
	Let $G_{m}$ be the Green's function of Problem \eqref{simp1} and suppose that $m>0$. Then $G_{m}$ is positive if, and only if, $m \in \left(0, \lambda_{1} \right)$, where $\lambda_{1}= \left(\frac{\pi}{2T} \right)^2$ is the smallest positive Dirichlet eigenvalue of Problem
	\begin{equation}
		z''(t)=-m\,z(-t), \quad t \in I, \quad z(-T)=z(T)=0.
		\label{autovalorm}
	\end{equation} 
\end{theorem}
 
\begin{proof}
	On the one hand, by Theorem \ref{signog}, we know that the largest value of $m$ for which the Green's function $H_{0,M}(=G_m)$ remains positive is given by $m= \left(\frac{\pi}{2T} \right)^2$.
	
	On the other hand, it is straightforward to verify that $m=\left(\frac{\pi}{2T}\right)^2$ corresponds to the first eigenvalue of Problem \eqref{autovalorm}.
	
	As a consequence, we deduce that the smallest eigenvalue $\lambda^{s_0}$ for the Dirichlet problems of Theorem \ref{teoautovalor} with $M=0$ is attained at $s_0=T$.
	
	Taking these observations into account, the proof of the theorem is complete.
\end{proof}

Let us now consider the case $m=0$.

%
%

With the help of Wolfram Mathematica, we can obtain the explicit expression of the first Dirichlet eigenvalue of problems \eqref{znoentero} and \eqref{zentero}, when $m=0$ (see Appendix \ref{code}). We observe (see Figure \ref{s48} for the case with $T=4.8$) that these eigenvalues decrease as $s_0$ increases, reaching their minimum at $s_0=T$. 
\begin{figure}[H]
	\centering
	\includegraphics[width=0.6\textwidth]{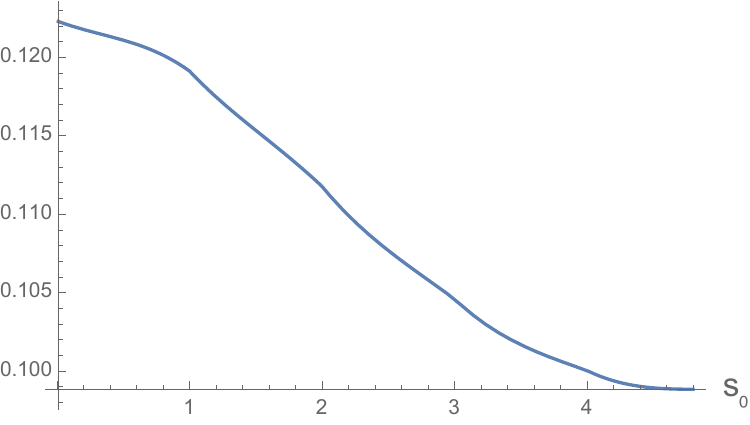}
	\caption{Graphical representation of the monotonic decrease of the first Dirichlet eigenvalue with respect to increasing values of $s_0$ for $T=4.8$.}
	\label{s48}
\end{figure}

For values of $T$ that are not too large, specifically for $T \leq 5$, the program developed in the Appendix \ref{code} allow us to obtain the explicit expression of the eigenvalues $\lambda^{s_0}$. By analyzing this expression, it is easy to prove analytically that the eigenvalue decreases as $s_0$ increases, which shows that for these values of $T$ the minimum is attained at $s_0=T$.

On the other hand, numerical experiments indicate that even for $T \geq 5$, the minimum still appears to be attained at $s_0=T$.

Based on all these observations, we conjecture that:
\begin{equation}
	\lambda_{1} = \min \left\{ \lambda^{s_0} \,\middle|\,
	\begin{array}{l}
		\lambda^{s_0} \text{ is the first eigenvalue of Problem \eqref{znoentero} with $m=0$ if } s_0 \in [0,T] \setminus \mathbb{N}, \\
		\text{or of Problem \eqref{zentero} with $m=0$ if } s_0 \in \mathbb{N} \cap [0,T]
	\end{array}
	\right\}=\lambda^{T}
\end{equation}
or, equivalently, $\lambda_{1}$ is the first Dirichlet eigenvalue of Problem
\begin{equation}
	z''(t)=-Mz([t]), \quad t \in I, \quad z(-T)=z(T)=0.
	\label{zsinxelo}
\end{equation}

 Again, with the help of Wolfram Mathematica, we plot the value of $\lambda_{1}$ as a function of $T$, obtaining the graph shown in Figure \ref{lambdat}. We also provide the explicit expression of $M$ in terms of $T$ for $T \leq 3$ (Table \ref{tab:M_values}). As the value of $T$ increases, the expression becomes progressively more complicated.

\begin{figure}[H]
	\centering
	\includegraphics[width=0.6\textwidth]{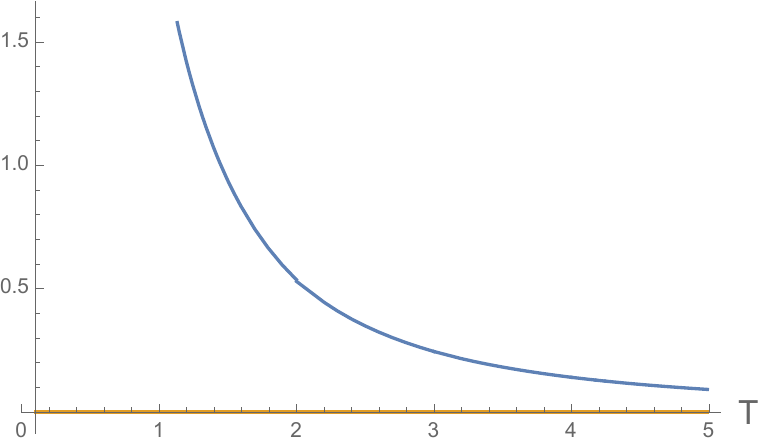}
	\caption{Representation of the eigenvalue $\lambda_1$ as a function of $T$.}
	\label{lambdat}
\end{figure}

\begin{table}[h!]
	\centering
	\begin{tabular}{|c|l|}
		\hline
		\textbf{Interval for $T$} & \textbf{Expression for $M$} \\
		\hline
		$0 < T < 1$ & $\displaystyle M = \frac{2}{T^2}$ \\
		\hline
		$1 < T < 2$ & $\displaystyle M = \frac{T^3 - \sqrt{-4 + 8T - 4T^2 + T^4}}{1 - 2T + T^2}$ \\
		\hline
		$2 < T < 3$ & 
		$\displaystyle M = \frac{1}{24(-2 + T)^2} \left( 208 + 32T(-7 + 2T) - \frac{8i(-i + \sqrt{3}) P(T)}{\Delta^{1/3}} + 8i(i + \sqrt{3}) \Delta^{1/3} \right)$ \\
		\hline
	\end{tabular}
	\caption{Expressions for $M$ depending on the interval of $T$. The definition of $\Delta$ and $P(T)$ is given below.}
	\label{tab:M_values}
\end{table}

\vspace{0.5em}

\noindent\textbf{Definition of $\Delta$ and $P(T)$:}
\vspace{1em}

\resizebox{\textwidth}{!}{$
	\begin{aligned}
		P(T) &= 169 + T(-364 + T(288 + T(-100 + 13T))), \\
		\Delta &= 2413 + T(-7530 + T(9762 + T(-6734 + T(2607 + T(-537 + 46T))))) \\
		&\quad + 3\sqrt{3} \sqrt{(-2 + T)^4 \cdot \big(2305 - T(7314 + T(-9600 + T(6680 + T(-2558 + T(446 + T(25 + 3(-8 + T)T)))))) \big)}
	\end{aligned}.
	$}

\vspace{1em}

Problem \eqref{tot1} with $m=0$ has been studied in several references, such as  \cite{buedo2024boundary, nieto2012second}.
However, in general, the constant sign of the Green's function has not been explicitly characterized. To the best of our knowledge, this is the first work providing a characterization in terms of the eigenvalue of a related Dirichlet problem. This result is particularly relevant, as knowing the exact sign of the Green's function allows the application of fixed-point theorems for the study of nonlinear problems.

On the other hand, by inspecting the graph in Figure \ref{lambdat}, it appears that the eigenvalue $\lambda^T$ decreases as the value of $T$ increases. We will prove this claim, but first we need to state some preliminary results. 

We begin by defining the concept of a positive operator, a cone and a semi-non-supporting operator.

\begin{definition}\cite[Definition 2.4]{Schaefer1974}
	Let $E$, $F$ be ordered vector spaces and let $K_T:E \rightarrow F$ be a linear operator. $K_T$ is called positive (in symbols: $K_T \geq 0$) if $K_T \,x>0$ for all $x>0$.
\end{definition}

\begin{definition}
	Given a Banach space $X$, $K \subset X$ is said to be a cone if it is a closed and convex subset of $X$ and satisfies the following properties.
	
	\begin{itemize}
		\item If $x \in K$, then $\lambda\,x \in K$ for all $\lambda \geq 0$.
		\item $K \cap (-K)=\{0\}$.
	\end{itemize}
\end{definition}

\begin{definition}\cite[Definition 1]{Sawashima1965}
	Let $X$ be a Banach space over $\mathbb{R}$ partially ordered by a cone $K$. A linear operator $K_T$ is a semi-non-supporting operator in $X$ with respect to $K$ if $K_T$ is positive and for each nonzero $x \in K$ and for each nonzero $f \in K^{*}$ (the dual space of $K$) there exists a natural number $n=n(x,f)$ such that $f(K_T^{n}x)>0$.
\end{definition}

The next theorem provides a characterization of semi-non-supporting operators.
\begin{theorem}\cite[Theorem 2]{Sawashima1965} Denoting $\rho(K_T)$ as the spectral radius of operator $K_T$. 
	Let $K_T$ have the resolvent $R(\lambda,K_T)$ with the point $\lambda=\rho(K_T)$ as its pole. Then, the following logical relation holds:
	\begin{equation*}
		K_T \text{ is semi-non-supporting operator} \Leftrightarrow \rho(K_T) > 0
		 \text{ and } K_T \text{ satisfies (A)},
	\end{equation*}
	where the condition $(A)$ is given by
	
	$(A)$ The spectral radius $\rho(K_T)$ is a simple eigenvalue of $K_T$ with a corresponding eigenvector in the interior of the cone $K$. The adjoint operator $K_T^*$ admits a strictly positive eigenfunctional associated with $\rho(K_T)$.
	\label{sawa}
\end{theorem}

Finally, we present the following theorem, which will be used in the proof.
\begin{theorem}\cite[Theorem 4.3]{marek1970frobenius}
Let $K \subset X$ be a closed cone, and let $K_{T_1}, K_{T_2}: X \to X$ be positive bounded linear operators. Let $K_{T_2}$ be a semi-non-supporting operator, let $\rho(K_{T_2})$ be an eigenvalue of $K_{T_2}$ and $K_{T_2}^{*}$ with an eigenvector $x_0 \in K$ and eigenfunctional $x_0' \in K^{*}$, and let $\rho(K_{T_1})$ be an eigenvalue of the adjoint operator $K_{T_1}^{*}$ with an eigenfunctional $x_1' \in K^{*}$. If 
\begin{equation*}
	K_{T_2}\ge K_{T_1}, K_{T_2} \neq K_{T_1},
\end{equation*}
then
\begin{equation*}
	\rho(K_{T_1})<\rho(K_{T_2}).
\end{equation*}
\label{marek}
\end{theorem}

Having established the above, we are now in a position to state the following result.
\begin{proposition}
	Let $\lambda^T$ denote the first (smallest) positive eigenvalue of Problem \eqref{zsinxelo} posed on the interval $I$. Then the map
	\begin{equation*}
		T \longrightarrow \lambda^T
	\end{equation*}
	is decreasing: if $0 <T_1<T_2$ then
	\begin{equation*}
		\lambda^{T_2} < \lambda^{T_1}.
	\end{equation*}
\end{proposition}

\begin{proof}
	Let $Gd_{T}$ denote the Green operator related to the problem
	\begin{equation}
		-u''(t)= \sigma(t), \quad t \in I, \quad u(-T)=u(T)=0,
		\label{ejgren}
	\end{equation}
	that is, $Gd_T:C(I) \rightarrow C(I)$ denotes the operator that assigns to each $\sigma \in C(I)$ the unique solution $u \in C^2(I) \subset C(I)$ of Problem \eqref{ejgren}. Moreover, let $B: C(I) \rightarrow PC(I)$ be the operator defined by $(Bu)(t):=u([t])$, where $PC(I)$ denotes the space of functions that are constant on each interval of the form $[[t],[t]+1)$.
	
	Then, the Problem \eqref{zsinxelo} can be written in operator form as 
	\begin{equation*}
		u=M \, Gd_{T}(Bu).
	\end{equation*}
	In particular, $M$ is an eigenvalue of the original problem if and only if $1/M$ is a nonzero eigenvalue of the operator $K_{T}:=Gd_{T} \circ B$. Since $Gd_{T}$ is compact and $B$ is continuous, the composition $K_{T}$ is a compact operator.
	
	From the expression of the Green's function $gd_{T}$ (the kernel of the operator $Gd_{T}$) related to Problem \eqref{ejgren}, it is easy to see that $gd_{T}$ is strictly positive for all $t,\, s \in (-T,T)$ and that it is strictly increasing with respect to the parameter $T$. Specifically, if $T_2>T_1$ then
	\begin{equation*}
		gd_{T_2}(t,s) > gd_{T_{1}}(t,s) \textup{ for all }t,s \in (-T_1,T_1).
	\end{equation*}
	Based on the above, we have $Gd_{T_2}$ restricted to $(-T_1,T_1)$ is greater than $Gd_{T_1}$ in the usual operator order.
	
	The operator $B$ preserves the positivity of the cone of non-negative functions. Thus, each $K_T$ is a compact and positive operator. 
	Then, by the Krein-Rutman Theorem \cite{KreinRutman}
	\begin{equation*}
		M(T)=\frac{1}{\rho(K_{T})}.
	\end{equation*}
	
	Furthermore, it is straightforward to verify that $K_T$ satisfies the conditions of Theorem \ref{sawa}, which implies that $K_T$ is a semi-non-supporting operator. In addition, the hypotheses of Theorem \ref{marek} are fulfilled, which implies that
\begin{equation*}
	\rho(K_{T_2}) > \rho(K_{T_1}).
\end{equation*}
As a consequence,
\begin{equation*}
	M(T_2)= \frac{1}{\rho(K_{T_2})} < \frac{1}{\rho(K_{T_1})}=M(T_1).
\end{equation*}
This proves that $M(T)$ is strictly decreasing with respect to $T$, completing the proof.
\end{proof}

As a corollary of Theorem \ref{teoautovalor}, the following result, valid for $T \leq 1$, is also obtained.

\begin{corollary}
	\label{coltmenor}
	Let $H_{m,M}$ be the Green's function related to Problem \eqref{tot1} and suppose that $m \geq 0$, $M \geq 0$, $m+M > 0$ and $T \in (0,1]$. Then $H_{m,M}$ is positive if, and only if, $M \in (-m, \lambda_{1})$ where $\lambda_{1}=\frac{m}{-1+\sec{\sqrt{m}T}}$ is the first Dirichlet eigenvalue of problem:
	\begin{equation}
		z''(t)=-mz(-t)-Mz([t]), \quad t \in I, \quad z(-T)=z(T)=0.
		\label{ausim}
	\end{equation}
\end{corollary}

\begin{proof}
	According to Theorem \ref{teoautovalor}, we know that $\lambda_1$ is the smallest Dirichlet eigenvalue among the following problems:
	\begin{equation}
		z''(t)=
		\left\{
		\begin{array}{lll}
			-m\,z(-t-2s_0)-M \,z(T-s_0), & \mbox{\rm if}  & t \in [-T,T-2s_0), \\
			-m\,z(-t-2s_0+2T)-M\, z(T-s_0), & \mbox{\rm if}  & t \in [T-2s_0,T]. \\
		\end{array}
		\right.
	\end{equation}
	We calculate the Dirichlet eigenvalues $\lambda^{s_0}$ associated with the aforementioned problem as a function of $s_0$, and find that
	
	\vspace{1em} 
	\resizebox{\textwidth}{!}{$
	\begin{aligned}
		\lambda^{s_0} = m \left( 
		- \frac{1}{
			\sinh\left( \sqrt{m} s_0 \right) \sin\left( \sqrt{m} T \right) 
			\text{csch}\left( \sqrt{m} T \right) \sec\left( \sqrt{m} (s_0 - T) \right) 
			\sinh\left( \sqrt{m} (s_0 - T) \right) 
			+ \cos\left( \sqrt{m} s_0 \right) - 1} - 1 \right)
	\end{aligned}
	$}
	\vspace{1em} 
	
	By analyzing the behavior of the expression for $\lambda^{s_0}$ and its derivative with respect to $s_0$, it can be seen that when $m \in \left( 0,\left(\frac{\pi}{2T}\right)^2\right)$, the eigenvalue $\lambda^{s_0}$ attains its minimum at $s_0=T$. Consequently, we obtain $$\lambda_1=\lambda^T=\frac{m}{-1+\sec{\sqrt{m}T}},$$
	which is the first Dirichlet eigenvalue of Problem \eqref{ausim}.
\end{proof}

\begin{remark}
	Numerical evidence suggests that Corollary \ref{coltmenor} may remain valid for values of $T>1$. Based on this observation, we propose the following conjecture:
	\begin{conjecture}
	Let $H_{m,M}$ be the Green's function of Problem \ref{tot1} and suppose that $m \geq 0$, $M \geq 0$, $m+M > 0$. Then $H_{m,M}$ is positive if, and only if, $M \in (-m, \lambda_{1})$ where $\lambda_{1}$ is the first Dirichlet eigenvalue related to problem:
	\begin{equation*}
		z''(t)=-mz(-t)-Mz([t]), \quad t \in I, \quad z(-T)=z(T)=0.
	\end{equation*}
	\end{conjecture}
\end{remark}

%
%
%

Having analyzed and characterized the regions of constant sign of the Green's function $H_{m,M}$, we now proceed to study the existence of solutions for the corresponding nonlinear problem.

\section{Existence of solution to nonlinear problems}
\label{monotono}
In this section, we will present some results that allow us to deduce the existence of solutions for nonlinear problems with periodic boundary conditions. Knowing the regions where the Green's function $H_{m,M}$ has a constant sign enables the direct application of various fixed-point theorems, such as the monotone method or Krasnosel'skii's theorem, among many others. Here, we focus on Krasnosel'skii's method.

\subsection{Krasnosel'skii's fixed-point method}

In this section, we will adapt the methods used in \cite{krasnoselskii1964positive} to prove the existence of solutions to differential equations, demonstrating new existence results for Problem 
\begin{equation}
	v''(t)=f(t,v(t), v(-t),v([t]) \quad \textup{a.e. } \quad t \in I, \quad v(-T)=v(T), \, v'(-T)=v'(T),
	\label{nolinear2}
\end{equation}
where $f:I \times \mathbb{R} \times \mathbb{R} \times \mathbb{R} \rightarrow \mathbb{R}$ is $2 \, T$-periodic in $t$ and a Carathéodory function. That is,

\begin{itemize}
	\item For a.e. $t \in I$ the function $(x,y,z) \in \mathbb{R} \times \mathbb{R} \rightarrow f(t,x,y,z) \in \mathbb{R}$ is continuous.
	\item For all $(x,y,z) \in \mathbb{R} \times \mathbb{R} \times \mathbb{R}$, the function $t \in I \rightarrow f(t,x,y,z)$ is Lebesgue-measurable.
	\item For all $R>0$, there exists $h_{R} \in \mathcal{L}^1(I)$ such that, if $|x|<R$ and $|y|<R$ then
	\begin{equation*}
		|f(t,x,y,z)| \leq h_{R}(t) \textup{ a.e }t \in I.
	\end{equation*}
\end{itemize}

%
%
%

We now state the fixed-point theorem that we will use \cite{krasnoselskii1964positive}.

\begin{theorem}
(Krasnosel'skii) Let $\mathcal{B}$ be a Banach space, and let $\mathcal{P} \subset \mathcal{B}$ be a cone in $\mathcal{B}$. Assume $\Omega_{1}$, $\Omega_{2}$ are open subsets of $\mathcal{B}$ with $0 \in \Omega_{1}$, $\Omega_{1} \subset \Omega_{2}$, and let $\mathcal{T}: \mathcal{P} \cap (\Omega_{2} \setminus \Omega_{1}) \rightarrow \mathcal{P}$ be a compact operator such that one of the following conditions is satisfied:

\begin{itemize}
\item[$(K_1)$]  $\quad \norm{\mathcal{T}v} \leq \norm{v}$ if $v \in \mathcal{P} \cap \partial \Omega_{1}$ and $\norm{\mathcal{T}v} \geq \norm{v}$ if $v \in \mathcal{P} \cap \partial \Omega_{2}$.
\label{cond1}
\item[$(K_2)$] $\quad \norm{\mathcal{T}v} \geq \norm{v}$ if $v \in \mathcal{P} \cap \partial \Omega_{1}$ and $\norm{\mathcal{T}v} \leq \norm{v}$ if $v \in \mathcal{P} \cap \partial \Omega_{2}$.
\label{cond2}
\end{itemize}

Then, $\mathcal{T}$ has, at least, one fixed-point at $\mathcal{P} \cap (\overline{\Omega_{2}} \setminus \Omega_{1})$.
\label{kranosel}
\end{theorem}

From now on, we will consider $m \in \mathbb{R}$, $M \in \mathbb{R}$, and $H_{m,M}$ as the Green's function related to the Problem

\begin{equation*}
\begin{aligned}
v''(t)+m\,v(-t)+Mv([t])=\sigma(t)&, \textup{ a.e. }t \in I \\
 v(-T)=v(T), \quad v'(T)=v'(-T)&,
\end{aligned}
\end{equation*}
where $T>0$, $m \in \mathbb{R}$ and $M \in \mathbb{R}$.

Let $L=\max{\{H_{m,M}(t,s): (t,s) \in I \times I\}}$ and $l=\min{\{H_{m,M}(t,s):(t,s) \in I \times I\}}$.

\begin{theorem}
Let $(m,M)$ be a pair such that the Green's function $H_{m,M}$ is positive on $I \times I$. Let us assume that there exist $r,R \in \mathbb{R}^{+}$, $r<R$, such that
\begin{equation}
f(t, x, y, z) + my + Mz \geq 0, \quad \forall \, x, y, z \in 
\left[ \frac{l}{L} r, \frac{L}{l} R \right], \textup{ a.e. } t \in I.
\label{primkra}
\end{equation}
Then, if any of the following conditions is satisfied:
\begin{enumerate}
\item 
\begin{align*}
f(t,x,y,z) + my + Mz &\geq \frac{L}{2Tl^{2}}\,x, \quad \forall x,y,z \in \left[\frac{l}{L}r, r \right], \, \, \, \textup{ a.e.. } t \in I, \\ 
f(t,x,y,z) + my + Mz &\leq \frac{1}{2TL}\,x, \quad \forall x,y,z \in \left[R, \frac{L}{l}R \right], \textup{ a.e. } t \in I,
\end{align*}
\item 
\begin{align*}
f(t,x,y,z) + my + Mz &\leq \frac{1}{2TL}\,x, \quad \forall x,y,z \in \left[\frac{l}{L}r, r \right], \, \, \, \textup{ a.e. } t \in I, \\ 
f(t,x,y,z) + my + Mz &\geq \frac{L}{2Tl^{2}}\,x, \quad \forall x,y,z \in \left[R, \frac{L}{l}R \right], \textup{ a.e.. } t \in I,
\end{align*}
\end{enumerate}
then Problem \eqref{nolinear2} has one positive solution on $I$.
\label{kraprincipal}
\end{theorem}

The proof follows directly the approach of \cite{torres2003existence}.

The following corollary of Theorem \ref{kraprincipal} is obtained by making the variable change $\hat{v}(t)=-v(t)$.

\begin{corollary}
Let $m$ and $M$ be a pair of values $(m,M)$ where the Green's function $H_{m,M}$ is positive. Suppose there exist $r, R \in \mathbb{R}^{+}$, $r<R$, such that
\begin{equation*}
f(t,x,y,z)+my+Mz \leq 0, \forall x,y,z \in \left[-\frac{L}{l}R, -\frac{l}{L}r \right], \textup{ a.e. }t \in I.
\end{equation*} 
Then, if any of the following conditions is satisfied,
\begin{enumerate}
\item 
\begin{align*}
f(t,x,y,z)+my+Mz &\leq \frac{L}{2Tl^{2}}\,x \quad \forall x,y,z \in \left[-r,-\frac{l}{L}r \right],  \, \, \, \textup{ a.e. }t \in I, \\ 
f(t,x,y,z)+my+Mz &\geq \frac{1}{2TL}\,x \quad \forall x,y,z \in \left[-\frac{L}{l}R,-R \right], \textup{ a.e. }t \in I, 
\end{align*}
\item 
\begin{align*}
f(t,x,y,z)+my+Mz &\geq \frac{1}{2TL}\,x \quad \forall x,y,z \in \left[-r,-\frac{l}{L}r \right], \, \, \, \textup{ a.e. }t \in I,  \\ 
f(t,x,y,z)+my+Mz &\leq \frac{L}{2Tl^{2}}\,x \quad \forall x,y,z \in \left[-\frac{L}{l}R,-R \right], \textup{ a.e. }t \in I, 
\end{align*}
\end{enumerate}
then Problem \eqref{nolinear2} has a negative solution on $I$.
\label{coaburrida}
\end{corollary}

%
%

To conclude, let's provide an example where we can ensure the existence of a solution to a problem using the results derived from Krasnosel'skii fixed-point Theorem.

%
%

\begin{example}
	We consider Problem \eqref{prorixi} under periodic boundary conditions. In this case, we have that
	\begin{equation*}
		f(t,x,y,z)=\frac{\alpha |z|^2x-\mu x+\beta y}{\hbar^2/2m_p}
	\end{equation*}
	and we will choose a pair $(m,M)$ such that 
	 $$m=-\frac{2\beta m_p}{\hbar^2}$$
and the Green's function $H_{m,M}$ related to Problem \eqref{prorixi} is positive (i.e. the pair $\left(\frac{-2\beta m_p}{\hbar^2},M \right)$ satisfies the properties of Theorem \ref{kraprincipal}). Then, according to Theorem \ref{kraprincipal}, if 
	\begin{equation*}
		\frac{\alpha |z|^{2}-\mu}{\hbar^{2}/2m_p}x+M\,z \geq 0 \quad \forall x,z \in \left[\frac{l}{L}r,\frac{L}{l}R\right]
	\end{equation*}
	and one of the following conditions holds
	\begin{enumerate}
		\item 
		\begin{align*}
			\left(\frac{\alpha |z|^{2}-\mu}{\hbar^{2}/2m_p}-\frac{L}{2Tl^{2}}\right)x+M\,z  &\geq 0, \quad \forall x,z \in \left[\frac{l}{L}r, r \right], \\ 
			\left(\frac{\alpha |z|^{2}-\mu}{\hbar^{2}/2m_p}-\frac{1}{2TL}\right)x+M\,z &\leq 0, \quad \forall x,z \in \left[R, \frac{L}{l}R \right],
		\end{align*}
		\item 
		\begin{align*}
			\left(\frac{\alpha |z|^{2}-\mu}{\hbar^{2}/2m_p}-\frac{L}{2Tl^{2}}\right)x+M\,z  &\leq 0, \quad \forall x,z \in \left[\frac{l}{L}r, r \right], \\ 
\left(\frac{\alpha |z|^{2}-\mu}{\hbar^{2}/2m_p}-\frac{1}{2TL}\right)x+M\,z &\geq 0, \quad \forall x,z \in \left[R, \frac{L}{l}R \right],
		\end{align*}
	\end{enumerate}

Problem \eqref{prorixi}, under periodic boundary conditions, has a positive solution on $I$.

If we consider the particular case where 
$$m=-\frac{2\beta m_p}{\hbar^{2}} \leq \left(\frac{\pi}{2 \, T} \right)^2 \textup{ and } M=0,$$ 
then, after some algebraic calculations and taking into account that $m_p$, $\mu$ and $\hbar$ must be positive, we obtain that if:

\begin{equation*}
	-\left(\frac{\pi}{2T} \right)^2 \frac{\hbar^2}{2m_p} \leq \beta \leq 0, \quad \alpha \geq \frac{\mu L^2}{l^2 r^2},
\end{equation*}
and one of the following conditions holds
	\begin{enumerate}
	\item 
	\begin{align*}
	\frac{L^2}{l^2 r^2} \left(\frac{\hbar^2 L}{4m_p T l^2}+\mu \right) \leq \alpha \leq \frac{l^2}{L^2 R^2} \left(\frac{\hbar^2}{4m_p T L}+\mu \right),
	\end{align*}
	\item 
	\begin{align*}
		 \frac{1}{R^2} \left(\frac{\hbar^2}{4m_p T L}+\mu \right) \leq \alpha \leq \frac{1}{ r^2} \left(\frac{\hbar^2 L}{4m_p T l^2}+\mu \right),
	\end{align*}
	for some $r>0$ and $R>0$ then, Problem \eqref{prorixi} under periodic boundary conditions has a positive solution on $I$.
\end{enumerate}

We could proceed in a similar way by applying Corollary \ref{coaburrida} to ensure the existence of a negative solution.

\end{example}

\section*{Appendix}

\subsection*{Numerical search for the maximum and the minimum of function $H_{m,M}$}
\label{apend}
\addcontentsline{toc}{section}{Appendix}

To infer the point $(\hat{t},\hat{s}) \in I \times I$ where the function $H_{m,M}$ reaches its minimum (maximum) when it is positive (negative), we developed a Python program to search, for different pairs of values $(m,M)$, the point where the function reaches its minimum (or maximum) when $m \geq 0$. We then plot the corresponding points for these values $(m,M)$. The resulting graph is shown in Figure \ref{figminmax}.

\begin{figure}[H]
\includegraphics[scale=.6, trim={0 8cm 0 7.6cm}, clip]{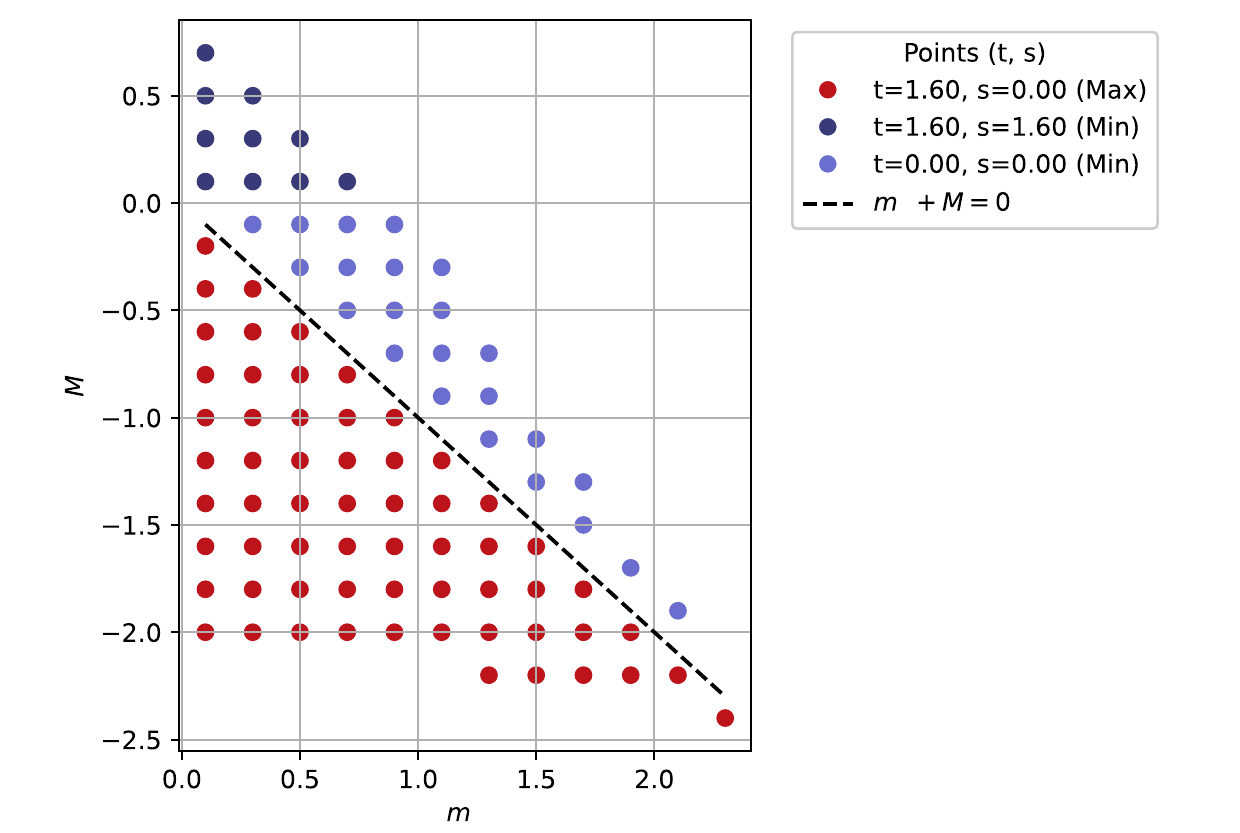}
%
%
\caption{Representation as a function of the values of $m$ and $M$ of the points where the function $H_{m,M}$ reaches its minimum (when positive) or maximum (when negative).}
\label{figminmax}       
\end{figure}

The maximum of the function $H_{m,M}$ seems to always occur at $(T,0) \in I \times I$, while the minimum, when the function is positive, occurs at $(T,T) \in I \times I$ or at $(0,0) \in I \times I$. By setting $H_{m,M}(T,T)=0$ and $H_{m,M}(0,0)=0$, and solving for $M$ as a function of $m$ we can plot both curves. We observe that for $m \in \left(0, \left(\frac{\pi}{2T}\right)^2\right)$, the function $H_{m,M}(T,T)$ is no longer positive for values of $M$ smaller than those for $H_{m,M}(0,0)$, and when $m \in \left(\left(\frac{\pi}{2T}\right)^2, \left(\frac{\pi}{2}\right)^2\right)$, the opposite happens. Therefore, we must choose the point $(T,T)$ when $m \in \left(0, \left(\frac{\pi}{2T}\right)^2\right)$ and $(0,0)$ when $m \in \left(\left(\frac{\pi}{2T}\right)^2, \left(\frac{\pi}{2}\right)^2\right)$.

A similar program was implemented for $m<0$.

\subsection*{Wolfram Mathematica code for the first Dirichlet eigenvalue}
\label{code}

Here we provide the Mathematica code required to symbolically compute the first Dirichlet eigenvalue of Problem \eqref{znoentero} with $m=0$ for $s_0 \in [0,T] \setminus \mathbb{N}$, and of Problem \eqref{zentero} with $m=0$ for $s_0 \in \mathbb{N} \cup [0,T]$.

\begin{lstlisting}[language=Python,caption={Mathematica code},label={lst:mathematicaM}]
	sSymbol = Symbol["s"];
	tSymbol = Symbol["TT"];
	n = 2*IntegerPart[T] + 2;
	Msols = ConstantArray[Function[{s, TT}, 0], IntegerPart[T] + 1];
	Soloxala = ConstantArray[0, IntegerPart[T] + 1];
	
	Do[
	With[{aFixed = a},
	A = Table[Symbol["A" <> ToString[i]], {i, 1, n}];
	B = Table[Symbol["B" <> ToString[i]], {i, 1, n}];
	v = ConstantArray[Function[{t, M}, 0], n];
	
	(* Define the first and last functions *)
	v[[n]] = Function[{t, M}, 
	-M/2*((A[[n]]*(TT + a - s) + B[[n]])/(1 + M/2*(TT + a - s)^2))*t^2 
	+ A[[n]]*t + B[[n]]
	];
	v[[1]] = Function[{t, M}, 
	-M/2*v[[n]][TT + a - s, M]*t^2 + A[[1]]*t + B[[1]]
	];
	
	(* Forward recursion *)
	Do[
	With[{jFixed = j}, 
	v[[j]] = Function[{t, M}, 
	-M/2*v[[jFixed - 1]][-TT + jFixed - 1 + a - s, M]*t^2 
	+ A[[jFixed]]*t + B[[jFixed]]
	];
	],
	{j, 2, 1 + IntegerPart[T] - a, 1}
	];
	
	(* Middle function *)
	v[[n - a]] = Function[{t, M}, 
	-M/2*((A[[n - a]]*(TT - s) + B[[n - a]])/(1 + M/2*(TT - s)^2))*t^2 
	+ A[[n - a]]*t + B[[n - a]]
	];
	
	(* Backward recursion *)
	Do[
	With[{jFixed = j}, 
	v[[j]] = Function[{t, M}, 
	-M/2*v[[jFixed + 1]][TT + jFixed - n + a - s, M]*t^2 
	+ A[[jFixed]]*t + B[[jFixed]]
	];
	],
	{j, n - a - 1, 2 + IntegerPart[T] - a, -1}
	];
	Do[
	With[{jFixed = j}, 
	v[[j]] = Function[{t, M}, 
	-M/2*v[[jFixed - 1]][TT + jFixed - (n - a) - s, M]*t^2 
	+ A[[jFixed]]*t + B[[jFixed]]
	];
	],
	{j, n - a + 1, n - 1, 1}
	];
	
	(* Derivatives *)
	vtt = Table[D[v[[j]][t, M], t], {j, 1, n}];
	vt = Function[{t, M}, #] & /@ vtt;
	
	(* Build continuity equations *)
	eqs = {};
	Do[
	AppendTo[eqs, v[[j]][-TT + j - s + a, M] == v[[j + 1]][-TT + j - s + a, M]];
	AppendTo[eqs, vt[[j]][-TT + j - s + a, M] == vt[[j + 1]][-TT + j - s + a, M]];
	, {j, 1, IntegerPart[T] - a}
	];
	AppendTo[eqs, v[[IntegerPart[T] + 1 - a]][-s, M] == v[[IntegerPart[T] + 2 - a]][-s, M]];
	AppendTo[eqs, vt[[IntegerPart[T] + 1 - a]][-s, M] == vt[[IntegerPart[T] + 2 - a]][-s, M]];
	
	Do[
	AppendTo[eqs, v[[j]][TT + j - n - s + a, M] == v[[j + 1]][TT + j - n - s + a, M]];
	AppendTo[eqs, vt[[j]][TT + j - n - s + a, M] == vt[[j + 1]][TT + j - n - s + a, M]];
	, {j, IntegerPart[T] + 2 - a, n - 1 - a}
	];
	Do[
	AppendTo[eqs, v[[j]][TT + j - n - s + a + 1, M] == v[[j + 1]][TT + j - n - s + a + 1, M]];
	AppendTo[eqs, vt[[j]][TT + j - n - s + a + 1, M] == vt[[j + 1]][TT + j - n - s + a + 1, M]];
	, {j, n - a, n - 1}
	];
	
	(* Boundary conditions *)
	AppendTo[eqs, v[[1]][-TT, M] == 0];
	AppendTo[eqs, v[[n]][TT, M] == 0];
	
	(* Solve the system *)
	matrizCoeficientes = CoefficientArrays[eqs, Join[A, B]][[2]];
	msimp = Simplify[matrizCoeficientes];
	pros = Det[msimp];
	Sols = Solve[pros == 0, M];
	Soloxala[[a + 1]] = Sols[[1]];
	
	(* Store the solution *)
	If[Length[Sols[[a + 1]]] > 0, 
	Msols[[a + 1]] = Function[{s, TT}, Re[M /. Soloxala[[aFixed + 1]] /. {sSymbol -> s, tSymbol -> TT}]], 
	Msols[[a + 1]] = Function[{s, TT}, Missing["NoSolution"]]
	];
	],
	{a, 0, IntegerPart[T]}
	]
\end{lstlisting}



\end{document}